\documentclass{article}[11]
\usepackage{amsmath, amssymb, color}

\newtheorem{thm}{Theorem}[section]
\newtheorem{prop}[thm]{Proposition}
\newtheorem{lem}[thm]{Lemma}
\newtheorem{cor}[thm]{Corollary}
\newtheorem{definition}[thm]{Definition}

\newenvironment{example}{\medskip\refstepcounter{thm}
\noindent{\bf Example \thesection.\arabic{thm}\ }}{\medskip}
\def\eq#1{{\rm(\ref{#1})}}
\newenvironment{proof}[1][,]{\medskip\ifcat,#1
\noindent{{\it Proof}:\ }\else\noindent{\it Proof of #1.\ }\fi}
{\hfill$\square$\medskip}
\newenvironment{remark}[1][Remark]{\begin{trivlist}
\item[\hskip \labelsep {\bfseries #1}]}{\end{trivlist}}

\DeclareMathOperator\vol{vol}

\DeclareMathOperator\Div{div}
\DeclareMathOperator\Diff{Diff}
\DeclareMathOperator\Ric{Ric}

\DeclareMathOperator\Vol{Vol}
\DeclareMathOperator\Ham{Ham}
\DeclareMathOperator\expG{exp}
\DeclareMathOperator\TRL{\mathcal{T}}
\DeclareMathOperator\tnabla{\widetilde{\nabla}}
\def\d{{\rm d}}
\def\w{\wedge}

\def\C{\mathbb{C}}
\def\CP{\mathbb{CP}}
\def\Sph{\mathbb{S}}
\def\S{\mathcal{S}}
\def\R{\mathbb{R}}
\def\Z{\mathbb{Z}}

\begin{document}

\title{Complexified diffeomorphism groups, totally real submanifolds and K\"ahler-Einstein geometry}

\author{Jason D. Lotay and Tommaso Pacini}


\maketitle

\begin{abstract}
Let $(M,J)$ be an almost complex manifold. We show that the infinite-dimensional space $\TRL$ of totally real submanifolds in $M$ carries a natural connection. 
This induces a canonical notion of geodesics in $\TRL$ and a corresponding definition of when a functional $f:\TRL\rightarrow \R$ is convex. 

Geodesics  in $\TRL$ can be expressed in terms of families of $J$-holomorphic curves in $M$; we prove a uniqueness result and study their existence. When $M$ is K\"ahler we define a canonical functional on $\TRL$; it is convex if $M$ has non-positive Ricci curvature. 

Our construction is formally analogous to the notion of geodesics and the Mabuchi functional on the space of K\"ahler potentials, as studied by Donaldson, Fujiki and Semmes. 
Motivated by this analogy, we discuss possible applications of our theory to the study of minimal Lagrangians in negative K\"ahler--Einstein manifolds.
\end{abstract}

\section{Introduction}
Let $(M,J)$ be a $2n$-dimensional manifold endowed with an almost complex structure. Given $p\in M$, we say an $n$-plane $\pi$ in $T_pM$ is \textit{totally real} if $J(\pi)\cap\pi=\{0\}$, 
\textit{i.e.}~if $T_pM$ is the complexification of $\pi$. An $n$-dimensional submanifold $L$ is \textit{totally real} if $T_pL$ is totally real in $T_pM$ for all $p\in L$. This gives a decomposition
$$T_pM=T_pL\oplus J(T_pL).$$ 
Although totally real submanifolds are a natural object in complex geometry,
 they cannot be studied using purely complex analytic tools. They are, in a
  sense, the opposite of complex submanifolds; in fact, they are  ``maximally non-complex'', where maximal also refers
to their dimension. Furthermore, the defining condition is an open one so their ``moduli space'' $\TRL$ is infinite-dimensional. 

It might seem reasonable to conclude that this class of submanifolds is too weak to carry interesting geometry. In this paper we will prove the contrary by initiating a study of the global geometric features of the space $\TRL$. Further results in this direction appear in the companion paper \cite{LP}; other applications appear in \cite{LPpersist}.

\paragraph{Geodesics on {\boldmath $\TRL$}.} Our first main result, described in Section \ref{ss:connection}, is that $\TRL$ admits a natural connection, inducing a notion of geodesics. In simpler language, we
 discover that there exists a notion of canonical 1-parameter deformations of a totally real submanifold $L$, in any given direction. This is rather striking: there is
 no analogue of this fact known in other spaces of submanifolds. In some sense this observation is the ``global version'' of the definition of totally real
 submanifolds, which says the ``normal'' space $T_pM/T_pL$ and tangent space $T_pL$ are canonically isomorphic via $J$. 
In other words, the extrinsic and intrinsic geometry of $L$ coincide; geodesics are, in a sense, the extrinsic analogue of the integral curves of tangent vector fields.

\paragraph{A convex functional.}  The geodesics induce a notion of convex functionals $f:\TRL\rightarrow\R$: specifically, those which are convex in one variable when restricted to each geodesic. A second striking fact is provided by the following example. Consider  $M=\C$, so that $\TRL$ is the space of curves: in this situation we prove that the standard length functional is convex in our sense. Interestingly, this turns out to be a reformulation of a classical result due to Riesz concerning certain convexity properties of integrals of the form $r\mapsto\int |u(re^{i\theta})|\,d\theta$, where $u$ is a subharmonic function on an annulus.

The length functional uses the metric on $\C$, so in higher dimensions it is natural to focus on K\"ahler, more generally almost Hermitian, manifolds $M$ and look for an analogous functional on $\TRL$. A first guess might be the standard Riemannian volume  functional but, in our context, this  is rather unnatural because it does not encode the totally real condition. In the literature \cite{Bor} one finds a second ``volume functional'', tailored specifically to totally real submanifolds. 

To understand this alternative functional, the key observation is that there exists a second, equivalent, definition of the totally real condition: $L$ is totally real if and only if the pullback operation for forms defines an isomorphism $K_{M|L}\simeq \Lambda^n(L;\C)$. One can view this as another manifestation of the ``extrinsic=intrinsic'' property of totally real submanifolds. When $L$ is oriented it turns out that ${K_M}_{|L}$ admits a canonical section. Integrating this (real) $n$-form on $L$ defines the ``$J$-volume functional'', which agrees with the length functional in dimension $1$ but is in general different to the Riemannian volume functional.

Our second main result, stated in Section \ref{ss:convexity}, is that, in the appropriate setting, this functional is convex in our sense. It is perhaps worth emphasizing that the notion of geodesics lies entirely within the realm of complex analysis: a priori, it has no relationship to K\"ahler geometry. Our convexity result thus reveals a new form of compatibility between complex and metric data.

\paragraph{Applications to minimal Lagrangian submanifolds.}This brings us to 
our study of the relationship between the $J$-volume and the  Riemannian volume.
The outcome is especially interesting when $M$ is a K\"ahler--Einstein (KE) manifold with negative scalar curvature. 

Recall that an $n$-dimensional submanifold $L$ in $M$ is Lagrangian if the ambient K\"ahler form  vanishes when restricted to $L$. Lagrangian submanifolds are a key topic in symplectic geometry. In the K\"ahler case it is particularly fruitful to study interactions between symplectic and Riemannian properties of $L$. For example, it is well-known that (i) in KE manifolds the mean curvature flow preserves the Lagrangian condition and (ii) in negative KE manifolds the minimal Lagrangians are strictly stable for the standard Riemannian volume.  

Fact (i) is the starting point for \cite{LP}. Our goal here is to further investigate fact (ii). 
Specifically, when $M$ is negative KE we show the following.
\begin{itemize}
 \item The $J$-volume provides a lower bound for the standard Riemannian volume. The two functionals coincide on Lagrangian submanifolds.\vspace{-4pt}
 \item The critical points of the $J$-volume are exactly the minimal Lagrangian submanifolds. It thus ``weeds out'' the additional critical points (non-Lagrangian minimal submanifolds) of the standard Riemannian volume.\vspace{-4pt}
 \item The $J$-volume is strictly convex with respect to our geodesics. For a minimal Lagrangian this is the global counterpart of the aforementioned infinitesimal stability property. 
\end{itemize}
It is thus clear that the $J$-volume provides good control over minimal Lagrangians. No such convexity holds for the Riemannian volume functional.

\paragraph{A moment map.}The above results fit into a larger picture. Indeed, the geometric features of $\TRL$ resemble those of two other  well-known infinite-dimensional spaces which appear in K\"ahler geometry: the integrable $(0,1)$-connections on a Hermitian vector bundle $E$, \textit{i.e.}~the holomorphic structures on $E$, and the K\"ahler potentials in a given K\"ahler class.  In both cases we have the following.
\begin{itemize}
 \item A canonical connection and notion of geodesics, related to an infinite-dimensional group action and its formal ``complexification''.\vspace{-4pt}
 \item A convex functional.\vspace{-4pt}
 \item A moment map encoding the group action, whose zero set coincides with the critical point set of the functional.
\end{itemize}
Following this lead, in Section \ref{s:grandconclusion} we show the geometry of $\TRL$ can be rephrased in terms of the formal complexification of the group
of (orientation-preserving) diffeomorphisms of $L$ and of a moment map induced by the $J$-volume functional. In particular, in the negative KE context it follows that minimal Lagrangians can be re-interpreted as the zero set of a moment map.

\paragraph{Open problems.} Our results naturally lead to  questions about minimal Lagrangians and their relationship with the geometry of negative KE manifolds.

In the analogous problem for K\"ahler potentials, the moment map serves to relate the existence of critical points of the functional to algebraic stability properties of the manifold, whilst the uniqueness of these points is related to the convexity of the functional. This formalism thus provides a useful understanding of the geometry of Fano manifolds, and was indeed one of the ingredients of the recently accomplished existence theory for positive KE metrics.

By contrast, the existence of KE-flat (Calabi--Yau) metrics was solved by Yau in the 1970s. Currently, the main questions here are related to calibrated geometry, mirror symmetry and its applications to String Theory in Physics. 

Given that the existence of negative KE metrics was  also solved in the 1970s, by Aubin and Yau, one might wonder what are the most interesting open questions in this context. Our results provide evidence that minimal Lagrangians are closely related to deep aspects of this geometry. They also show that the $J$-volume is a useful tool with which to probe this relationship.

On a technical level, a key feature of the space of K\"ahler potentials was its amenability to analytic methods. This led (across 20 years) to a complete existence theory for geodesics and to the corresponding extension of convexity results. The main analytic question we set up in this paper is whether an analogous theory is possible for geodesics in $\mathcal{T}$. In Section \ref{s:geodesic_eq} we provide a reformulation of the geodesic equation in terms of families of $J$-holomorphic curves intersecting the initial submanifold $L$. In the holomorphic setting this helps elucidate key features of the equation, by allowing us to use standard techniques from one complex variable to build examples and counterexamples to the existence of solutions. It also provides a fairly complete understanding of the uniqueness problem for geodesics. It is clear however that the final answer to these questions will require substantial effort, on a different technical scale.
More generally, it seems worthwhile investigating the properties of geodesics in relation to other classical problems in complex analysis. After this work was complete it was pointed out to us by L\'aszl\'o Lempert that the notion of geodesics, in the 1-dimensional case, had already been investigated in unpublished work by Birgen \cite{Birgen} in relation to Levi-flat hypersurfaces and polynomial hulls. Work in progress by Maccheroni \cite{roberta} shows that the notion of geodesics also finds applications to the study of complex-analytic properties of minimal Lagrangians.

A second significant problem is the existence and uniqueness of minimal Lagrangians in negative KE manifolds. As for K\"ahler potentials, existence may be related to a stability-type condition on the given data while our convexity result provides some information on global uniqueness properties, cf. Section \ref{s:grandconclusion}. Other aspects of the uniqueness question are discussed in  \cite{Joy} and in \cite{LPpersist}.

\ 

\textit{Thanks to} Filippo Bracci, Robert Bryant, Simon Donaldson, Jonny Evans, Pavel Gumenyuk, Dominic Joyce, Claude LeBrun, L\'aszl\'o Lempert, Fulvio Ricci and Chris Wendl for useful conversations and invaluable insights.

\section{The space of totally real submanifolds}\label{s:totally_real}
To start, let us make three initial choices:
\begin{itemize}
 \item a $2n$-dimensional manifold $(M,J)$ with an almost complex structure;\vspace{-4pt}
 \item an oriented $n$-dimensional manifold $L$;\vspace{-4pt}
 \item a totally real immersion $\iota:L\rightarrow M$.\footnote{This choice serves only to determine the homotopy class of immersions we will study.}
\end{itemize}

It will be important to maintain the distinction between immersions of $L$ and their corresponding images, \textit{i.e.}~``submanifolds''. 
In general, a \textit{submanifold} is an equivalence class of immersions, up to reparametrization via a diffeomorphism of $L$.  
Since the orientation of $L$ will play a role, we are interested in a slightly more refined notion: 
an \textit{oriented submanifold} is an equivalence class of immersions, up to reparametrization by orientation-preserving diffeomorphisms. 
 
The totally real condition is preserved under reparametrization, so it is well-defined on the space of (oriented) submanifolds. 
 We now define our two main spaces of interest.
 
 \begin{itemize}
  \item Let $\mathcal{P}$ be the space of totally real immersions of $L$ into $M$ which are homotopic, through totally real immersions, to the given $\iota$.\vspace{-4pt}
  \item Let $\mathcal{T}$ be the space of oriented totally real submanifolds  obtained as the quotient of $\mathcal{P}$ by the group\footnote{To simplify notation we omit any reference to the orientation.} $\mbox{Diff}(L)$ of orientation-preserving diffeomorphisms of $L$.
 \end{itemize}
 
We shall view $\pi:\mathcal{P}\rightarrow \TRL$, where $\pi$ is the natural projection, as a principal fibre bundle with respect to the obvious right group action of $\mbox{Diff}(L)$. 
The totally 
real condition is open in the Grassmannian of tangent $n$-planes in $M$, so it is a ``soft'' condition: 
in particular, $\mathcal{P}$ is an open subset of the space of all immersions. It thus has a natural Fr\'echet structure, making it an infinite-dimensional manifold. 
Given any $\iota\in\mathcal{P}$, we can identify $T_\iota\mathcal{P}$ with the space of all sections of (the pull-back of) the bundle $TM$ over $L$.

Moreover, $\TRL$ is (at least formally) also an infinite-dimensional manifold. Given $L\in \TRL$, its tangent space $T_L\TRL$ can be obtained via the infinitesimal analogue of the operation which quotients immersions by reparametrization. Specifically, $T_L\TRL$ can be identified with sections of the bundle $TM/TL\simeq J(TL)\simeq TL$; we conclude that $T_L\TRL\simeq \Lambda^0(TL)$. The key point is that the totally real condition provides a canonical subspace in $TM$  
transverse to $TL$ \emph{and} a canonical isomorphism of this space with $TL$; \textit{i.e.}~the (extrinsic) ``normal'' bundle (defined via quotients) is canonically isomorphic to the (intrinsic) tangent bundle.

\begin{remark} The action of $\mbox{Diff}(L)$ might not be free; it is guaranteed to be free only for embeddings. We will not worry about this issue, 
just as we will not be concerned about precise definitions of infinite-dimensional manifolds, Lie groups and bundles. Everything concerning such matters is taken as   purely formal, 
but it provides vital insight into the geometry of $\TRL$. We refer to \cite{infdim} for one approach to infinite-dimensional geometry and analysis which could be applied here.
\end{remark}

\begin{remark} Some orientable manifolds, \textit{e.g.}~$n$-spheres $
\mathbb{S}^n$, admit an orientation-reversing diffeomorphism $\phi$. In this case, reparametrization by $\phi$ defines a natural $\Z_2$-action on the space of immersions; two initial choices of totally real immersion related this way define  different (non-homotopic) spaces $\mathcal{P}$, thus $\TRL$. Other orientable manifolds do not admit such diffeomorphisms: \textit{e.g.}~$\CP^2$. In this case there is no distinction between submanifolds and oriented submanifolds.
 \end{remark}

\subsection{A canonical connection and geodesics}\label{ss:connection}

Differentiating the action of $\mbox{Diff}(L)$ at $\iota\in\mathcal{P}$ we obtain a subspace $V_\iota$ of the tangent bundle $T_\iota\mathcal{P}$, canonically isomorphic to the Lie algebra $\Lambda^0(TL)$ of vector fields on $L$. The space $V_\iota$ is the kernel of $\pi_*[\iota]:T_\iota\mathcal{P}\rightarrow T_{\pi(\iota)}\TRL$ and is given by
$$V_{\iota}=\{\iota_*X:X\in\Lambda^0(TL)\}.$$ 

Consider 
$$H_\iota:=J(V_\iota)=\{J\iota_*X:X\in \Lambda^0(TL)\}.$$
This space gives a complement to $V_\iota$, in the sense that there is a decomposition 
$$T_\iota\mathcal{P}=V_\iota\oplus H_\iota.$$
Varying $\iota$ in $\mathcal{P}$ we obtain a distribution $H$ in $T\mathcal{P}$. 

Let $\varphi\in\Diff(L)$ and let $\iota\in\mathcal{P}$.  Let $R_{\varphi}$ denote the right action of $\varphi$ on $\mathcal{P}$, \textit{i.e.}~$R_{\varphi}\iota=\iota\circ\varphi$.    We now show that the distribution $H$ is right-invariant.

\begin{lem}\label{l:H_defines_connection}
Let $\varphi\in\Diff(L)$ and $\iota\in\mathcal{P}$.  Then
$(R_{\varphi})_*H_{\iota}=H_{R_{\varphi}\iota}$.
\end{lem} 

\begin{proof}
Let $J\iota_*X\in H_{\iota}$.  Then $J\iota_*X\in T_{\iota}\mathcal{P}$, so by definition
 there exists a curve $\iota_t$ in $\mathcal{P}$ with $\iota_0=\iota$ and $\frac{\d\iota_t}{\d t}|_{t=0}=J\iota_*X$.  Thus we may calculate for $p\in L$:
\begin{align*}
(R_{\varphi})_*J\iota_*|_pX|_p&=\frac{\d}{\d t}(R_\varphi\circ\iota_t)|_{t=0,p}=\frac{\d}{\d t}(\iota_t\circ\varphi)|_{t=0,p}=\frac{\d\iota_t}{\d t}|_{t=0,\varphi(p)}\\
&=J\iota_*|_{\varphi(p)}X|_{\varphi(p)}=J\iota_*|_{\varphi(p)}\circ\varphi_*|_p\circ\varphi_*^{-1}|_{\varphi(p)}X|_{\varphi(p)}\\
&=J(\iota_*\circ\varphi_*)|_p(\varphi_*^{-1}X)|_{p}.
\end{align*} 
Hence $(R_{\varphi})_*J\iota_*X=J(R_\varphi\iota)_*(\varphi_*^{-1}X)\in H_{R_{\varphi}\iota}$.
\end{proof}

\noindent By Lemma \ref{l:H_defines_connection} and the general theory of principal fibre bundles, $H$ defines a connection on the principal fibre bundle $\mathcal{P}$.

Recall  from the general theory that any representation $\rho$ of $\mbox{Diff}(L)$ on a vector space $E$ defines an associated vector bundle $\mathcal{P}\times_\rho E$ over $\TRL$; each fibre of this bundle is isomorphic to $E$. Such a bundle has an induced connection. Parallel sections of this bundle can be described as follows.
Choose a curve of submanifolds $L_t$ in $\TRL$. Choose a horizontal lift $\iota_t$, \textit{i.e.}~a curve in $\mathcal{P}$ satisfying $\pi(\iota_t)=L_t$ and $\frac{\d}{\d t}\iota_t\in H_{\iota_t}$. Choose any ($t$-independent) vector $e\in E$. Then the section $[(\iota_t,e)]$ of $\mathcal{P}\times_\rho E$, 
defined along $L_t$, is parallel. We can obtain all such parallel sections simply by varying $e$.

In particular, using the adjoint representation of $\mbox{Diff}(L)$ on its Lie algebra gives the vector bundle $\mathcal{P}\times_{\text{ad}}\Lambda^0(TL)$. 
It is of fundamental importance to us that this bundle is canonically isomorphic to the tangent bundle of $\TRL$, via 
\begin{equation}\label{iso.tgt.bundle}
 \mathcal{P}\times_{\text{ad}}\Lambda^0(TL)\simeq T(\TRL),\ \ [\iota,X]\mapsto \pi_*[\iota](J\iota_*X).
\end{equation}
\begin{remark}
When $M$ is complex (so $J$ is integrable), we revisit \eqref{iso.tgt.bundle} in 
Sections \ref{ss:cpx_lie} and \ref{ss:inf_complexification} from another 
viewpoint, as a consequence of Proposition \ref{p:homogeneous}. 
\end{remark}

The isomorphism \eqref{iso.tgt.bundle} implies that the connection given by $H$ on $\mathcal{P}$ 
induces a connection on $T(\TRL)$. We can then describe parallel vector fields on $\TRL$ as above. Finally, recall that a curve $L_t$ is a geodesic if its tangent vector field 
$\frac{\d}{\d t}(L_t)$ is parallel. We thus obtain the following characterization of geodesics in $\TRL$.

\begin{lem}\label{geod.lem} A curve $L_t$ in $\TRL$ is a geodesic if and only if there exists a curve of immersions $\iota_t$ and a fixed vector field $X$ in $\Lambda^0(TL)$ such that $\pi(\iota_t)=L_t$ and
\begin{equation}\label{geod.eq}
\frac{\d}{\d t}\iota_t=J\iota_{t*}(X).
\end{equation}
This implies that $[\iota_{t*}X,J\iota_{t*}X]=0$, for all $t$ for which $L_t$ is defined.
\end{lem}
\begin{proof}
 The form of \eqref{geod.eq} proves $\iota_t$ is horizontal, and $X\in\Lambda^0(TL)$ plays the role of $e\in E$ in the general theory.
 Assume $L_t$ is a geodesic defined for $t\in (-\epsilon,\epsilon)$.
Let $x(s)$ be an integral curve of $X$ on $L$, defined for some $s\in (a,b)$. Then 
$$f(s,t):(a,b)\times (-\epsilon,\epsilon)\rightarrow M,\ \ f(s,t)=\iota_t(x(s))$$  
is an immersed surface in $M$ and $\iota_{t*}X$, $J\iota_{t*}X$ represent its coordinate vector fields in the $s$ and $t$ directions, respectively. As such, they commute.
\end{proof}

 \begin{remark}
 The existence of a canonical connection on the space of totally reals  appears to be rather surprising. 
 One might wonder why this is not true for the space $\mathcal{S}$ of all submanifolds.  
 Although one can show there is a canonical right-invariant horizontal distribution on the space of all immersions $\mathcal{I}$, 
 defined by sections of the normal bundle, 
 one seems unable to view $T(\mathcal{S})$ as a vector bundle associated to $\mathcal{I}$, so it does not receive an induced connection.  
 In other words, the group action on $\mathcal{I}$ encodes only intrinsic information, 
 and in general one cannot encode the extrinsic geometry of the normal bundle intrinsically.
 \end{remark}

 \subsection{Convexity}
 
 Given geodesics, one  has a natural definition of convex functionals on $\TRL$.

\begin{definition}\label{convex.dfn}
A functional $F:\TRL\rightarrow\R$ is convex if and only if it restricts to a convex function in one variable along any geodesic in $\TRL$.
\end{definition}

\noindent In the absence of existence results for geodesics, this notion could be vacuous. However, in the presence of geodesics, convex functionals provide 
powerful tools for analysing the geometry of $\TRL$. We thus now turn to the existence problem.

 \section{The geodesic equation}\label{s:geodesic_eq}
 
Once one has a notion of geodesics on a manifold $\mathcal{M}$, there are two key existence issues which arise: (i) the Cauchy problem, \textit{i.e.}~the short-time existence of geodesics given an initial point and direction, and (ii) the boundary value problem, \textit{i.e.}~the existence of geodesics between two points in $\mathcal{M}$. 
 
 When $\mathcal{M}$ is finite-dimensional, or infinite-dimensional and Banach, the first problem is purely local and can be solved via the standard existence theory for ordinary differential equations. The second problem concerns the global properties of $\mathcal{M}$ and relates to the definition of geodesic completeness.
 
 In our case the manifold $\TRL$ is infinite-dimensional but only Fr\'echet, so existence and uniqueness results for geodesics are non-trivial. The goal of this section is to rephrase our notion of geodesics in terms of families of $J$-holomorphic curves in $(M,J)$. This has several advantages.
 \begin{itemize}
  \item It offers a geometrically appealing reformulation of the geodesic equation.\vspace{-4pt}
  \item It clarifies the nature of the geodesic equation, indicating for example that it is not elliptic; however, it can be written as a family of elliptic equations.\vspace{-4pt}
  \item It opens the door to standard tools in the theory of one complex variable.
 \end{itemize}
This viewpoint will lead us, at least when $M$ is complex, to a complete solution of the uniqueness question. It does not give a complete answer to the existence problem, but it does yield useful insight by providing both examples and counterexamples and by suggesting a slight weakening of the notion of solution.
 
 \subsection{A reformulation of the geodesic equation}\label{ss:reformulation}
 We distinguish three cases: the Cauchy problem, geodesic rays and the boundary value problem.
\paragraph{The Cauchy problem.}Assume we have an initial  $L_0\in\TRL$ and  initial direction in $T_{L_0}\TRL$, which may be identified with a smooth vector field $X$ on the abstract manifold $L$. Ideally, the initial value problem for the geodesic equation \eqref{geod.eq} can then be solved as follows.

\begin{enumerate}
 \item Choose an initial parametrization $\iota_0$ of the submanifold $L_0$.\vspace{-4pt}
 \item Consider the flow defined by $X$ on $L$; choose an integral curve $x=x(s)$ of $X$, where $s\in I:=(a,b)$.\vspace{-4pt}
 \item Seek a $J$-holomorphic curve $\iota(s,t):I\times (-\epsilon, \epsilon)\rightarrow (M,J)$ such that $\iota(s,0)=(\iota_0\circ x)(s)$. 
Here, $I\times (-\epsilon, \epsilon)$ has its standard complex structure. \vspace{-4pt}
 \item Varying the integral curve $x$ gives a family of $J$-holomorphic curves $\iota$. Since each point of $L$ belongs to some integral curve, fixing the time parameter $t$ defines a map $\iota_t:L\rightarrow M$ which coincides with $\iota_0$ for $t=0$.
\end{enumerate}
If the $J$-holomorphic curves depend smoothly on the integral curves, $\iota_t$ will be smooth. Since immersions form an open set in the space of maps, $\iota_t$ will be an immersion for small $t$. Finally, the $\iota_t$ solve \eqref{geod.eq} by construction.

Though appealing, this procedure entails some difficulties. In particular, we observe the following.

\begin{itemize}
\item Fix $X$. To obtain a map defined on $L$ for each $t\in (-\epsilon, \epsilon)$ we need to be able to choose $\epsilon$ independent of the integral curve.\vspace{-4pt}
\item Moreover, we would like to find appropriate ``uniformly bounded" vector fields such that $\epsilon$ is independent of the specific $X$. This is related to the possibility of defining an ``exponential map" from a ball in $T_{L_0}\TRL$ into $\TRL$.
\end{itemize}

To proceed we must determine the correct framework within which to analyze our $J$-holomorphic equations. 
We are trying to solve an elliptic problem on $I\times(-\epsilon, \epsilon)$ by prescribing data inside the domain rather than, say, on the boundary. 
Notice that the domain itself is not prescribed as $\epsilon$ is to be determined. 
To tackle this problem, it is natural to use the ``method of characteristics''. 
The initial data is assigned on the curve  $I\times\{0\}\subset I\times(-\epsilon,\epsilon)$: this curve is non-characteristic for our equation, so the method makes sense. 

Here, the only general existence result available is the Cauchy--Kowalevski theorem, which requires real analytic initial data. This 
regularity restriction is rather strong: from the geometric viewpoint one wants geodesics in the space of smooth immersions, built as 
above using maps $C^\infty\big(I\times (-\epsilon,\epsilon),M\big)$. 
On the other hand, when $M$ is complex, standard regularity theory implies that any solution is complex analytic with respect to  $s+it$. 
In particular, if the solution exists, the initial data $\iota_0\circ x(s)$ must be real analytic. 
We conclude that the analytic setting is actually natural for the geodesic problem stated above, where the Cauchy--Kowalevski theorem provides strong existence results in Theorem \ref{thm:exp}. 

This same reasoning also demonstrates an obstruction to the existence of solutions to the 
 Cauchy problem when the initial data is only assumed to be smooth. It is thus important to introduce a weaker notion of geodesic, as follows. 

 \paragraph{Geodesic rays.}
 
 In the standard theory of one complex variable, one often studies maps defined on closed domains: holomorphic on the interior, but only smooth or continuous up to the boundary. This leads us to the following.
 
 \begin{definition} \label{def:geodesic_ray}
 Fix $L_0\in\TRL$ and $JX\in T_{L_0}\TRL$. A \textit{geodesic ray} starting from $L_0$ with direction $JX$ is a curve of submanifolds $L_t$ in $\TRL$, for $t\in [0,\epsilon)$, for which there exists a curve of immersions $\iota_t$, for $t\in [0,\epsilon)$, with the following properties:
 \begin{itemize}
 \item $\iota_t$ is smooth on $L\times[0,\epsilon)$;\vspace{-4pt}
  \item for $t\in (0,\epsilon)$, $\iota_t$ solves the geodesic equation \eqref{geod.eq};
  \vspace{-4pt}
  \item $\iota_0$ parametrizes $L_0$.
 \end{itemize}
 \end{definition}
 
\noindent The existence problem for geodesic rays  is manifestly different from the Cauchy problem previously described.

\paragraph{The boundary value problem.} We can  now define  geodesics between two submanifolds $L_0$ and $L_1$ in $\TRL$ as geodesic rays interpolating between them. 

To prove the existence of such a geodesic it is necessary to find a vector field $X$ on $L$ and smooth, totally real immersions $\iota_t:L\rightarrow M$, for $t\in [0,1]$, so that:
\begin{itemize}
\item $(x,t)\mapsto \iota_t(x)$ is smooth on $L\times [0,1]$;\vspace{-4pt}
 \item for $t\in (0,1)$, $\iota_t$ solves \eqref{geod.eq};
 \vspace{-4pt}
 \item $\iota_0$, $\iota_1$ parametrize $L_0$, $L_1$.
  \end{itemize}
 As above, we can decompose a geodesic ray into a family of $J$-holomorphic
  curves parametrized by the integral curves of $X$ on $L$, thus defined on domains $I\times [0,\epsilon)$. For the boundary value problem, each  curve provides a $J$-holomorphic filling between the boundary data prescribed by $\iota_0$ on $I\times \{0\}$ and $\iota_1$ on $I\times \{1\}$.

 \begin{remark}
If $X$ defines a geodesic ray for $t\in [0,\epsilon)$ then $-X$ defines a geodesic ray for $t\in (-\epsilon,0]$ and the two induced families $L_t$ coincide, up to time reversal.
 \end{remark}

\subsection{Existence in the real analytic case}\label{ss:Cauchy_Kov}

The goal of this section is to prove the existence of an ``exponential map" on $\mathcal{T}$  in the real analytic context, with respect to a Fr\'echet-type metric, as follows.

\begin{thm}\label{thm:exp}
Let $L$ be a compact real analytic $n$-manifold, let $(M,J)$ be a real analytic almost complex $2n$-manifold such that $J$ is also real analytic, and let 
$\iota_0:L\to M$ be a real analytic, totally real immersion.  
Fix $m, R>0$ and let $\mathcal{B}(m,R)$ be the space of real analytic vector fields $X$ on $L$ with\footnote{Here, $\nabla$ and the norm are defined with respect to some choice of metric on $L$: since $L$ is compact, all metrics are equivalent.} 
\begin{equation}\label{exp.bounds}
\|\nabla^kX\|_{C^0}\leq mk!/R^k\ \ \text{for all $k\geq 0$.}
\end{equation} 
There exists $\epsilon>0$ (depending on $m,R$) such that, for each $X\in\mathcal{B}(m,R)$, there is a geodesic $(L_t)_{t\in(-\epsilon,\epsilon)}$ in $\mathcal{T}$ given by immersions $\iota_t:L\to M$ satisfying \eqref{geod.eq}
 and $\iota_t|_{t=0}=\iota_0$.
\end{thm}

In the above generality, the proof is an application of the Cauchy-Kovalewski theorem (see e.g.~\cite[Chapter 10 Theorem 4]{SpivakVolV}).  The key ingredient in this theorem goes under the name ``method of majorants". If $J$ is integrable the proof of Theorem \ref{thm:exp} is more transparent, and the non-integrable case works with the same method. We thus limit ourselves to the integrable setting.

To simplify, identify $L$ with its image $\iota_0(L)\subset M$. For each $p\in L$, choose an open polydisk $P_i\subseteq\C^n$ serving as a holomorphic coordinate chart for $M$, such that  $V_i:=P_i\cap\R^n$ is a coordinate chart for $L$. Then choose $U_i\ni p$ which is open and compactly contained in $V_i$. By compactness of $L$ we can extract a finite number of domains so that the $U_i$ cover $L$. We now proceed in two steps.
\begin{enumerate} 
\item Given $m,R>0$ we find $\epsilon>0$ such that, for $X$ satisfying \eqref{exp.bounds} and $x_0\in U_i$, there exists a unique real analytic integral curve $x:(-\epsilon,\epsilon)\to V_i$ of $X$ with $x(0)=x_0$.\vspace{-4pt}
\item We then show that each complexified power series $x(s+it)$, for $|s+it|<\epsilon$, takes values in $P_i$. Up to identifications, this allows us to define $\iota_t(s):=x(s+it)$ for $|s|<\epsilon/\sqrt{2}$, $|t|<\epsilon/\sqrt{2}$. By varying $x_0\in U_i$ we see that $\iota_t$ is well-defined on $L$, and satisfies \eqref{geod.eq} by construction. 
\end{enumerate}
The integral curve equation is a system of ODEs of the form $\dot{x}(s)=X(x(s))$. The existence of a unique smooth solution follows from standard ODE theory.  For the proof of Step 1, we review how the method of majorants shows that this solution is real analytic and furnishes a radius of convergence of the corresponding power series. It suffices to focus on the scalar ODE case. 

Suppose we have open sets $0\in U\subset V\subseteq \R$ with $U$ compactly contained in $V$. Assume we want to solve the scalar ODE:
\begin{equation}\label{scalar.ODE}
\dot{x}(s)=f(x(s)),\ \ x(0)=x_0\in U,
\end{equation}
where $f$ is any real analytic function defined on $V$ satisfying 
\begin{equation}\label{exp.bounds.2}
\|f^{(k)}\|_{C^0(U)}\leq mk!/R^k \ \ \text{for all $k\geq 0$;}
\end{equation} equivalently, for any $x_0\in U$ the coefficients of the power series representation  $f(x)=\sum a_n(x-x_0)^n$ satisfy $a_n\leq m/R^n$. 
Recall the following definition.

\begin{definition} The power series $A(x)=\sum A_n x^n$ is a majorant of the power series $a(x)=\sum a_n x^n$, and we write $a<<A$, if $|a_n|\leq A_n$ for all $n\geq 0$.
\end{definition}
If $a<<A$ it follows that (i) if $A$ converges with radius $R_A$ then $a$ converges with radius $R_a\geq R_A$, and (ii) if we fix $\epsilon\in (0,R_A)$ we can uniformly bound the values of $a$: $|a(x)|\leq \sum|a_n||x|^n\leq \sum A_n\epsilon^n=A(\epsilon)$, for all $|x|\leq \epsilon$.

Consider first $x_0=0$. The  bounds \eqref{exp.bounds.2} on $f$ show that, if $F(x)=mR/(R-x)$ then the power series of $f$ based at $0$ satisfies $f<<F$. The method of majorants  shows, by examining the induced equations on the higher derivatives, that the power series of the solution $x(s)$ of \eqref{scalar.ODE} satisfies $x<<\xi$, where $\xi$ solves
\begin{equation}\label{eq:majorant}
\dot{\xi}(s)=F(\xi(s)), \ \ \xi(0)=0.
\end{equation} 
For arbitrary $x_0\in U$ consider the power series of $f$ based at $x_0$: $f(x)=\sum a_n(x-x_0)^n$. If we set $y(s):=x(s)-x_0$ so that $y(0)=0$ we find $\dot{y}=\sum a_ny^n$. The bounds \eqref{exp.bounds.2} imply that this equation for $y$ can again be compared with \eqref{eq:majorant}, so that $y<<\xi$. We conclude that $x<<\xi+x_0$.

 Equation \eqref{eq:majorant} can be explicitly solved, which gives the radius of convergence of $\xi$ in terms of $m,R$. We thus obtain (i) a lower bound on the radius of convergence of $x(s)$, for any $f$ satisfying \eqref{exp.bounds.2} and initial data $x_0\in U$, and (ii) an upper bound on the values of $|x(s)|$ for $s\in(-\epsilon,\epsilon)$ in terms of $\xi(\epsilon)+x_0$. In particular, since $U$ is compactly contained in $V$, by restricting $\epsilon$ we may assume that all solutions $x(s)$, for $x_0\in U$, are contained in $V$. Step 1 can be proved by applying the same reasoning to $\dot{x}=X(x)$ in each coordinate chart $V_i$. 

Our assumption that $J$ is integrable allows us to complexify this data by simply complexifying the corresponding power series; if $J$ were only almost-complex this is where we would use the Cauchy--Kovalewski theorem to prove the existence of such complexified data, \textit{i.e.} to obtain solutions $x(s,t)$ of the equation 
$$\frac{\partial x}{\partial t}_{|(s,t)}=J(x)\frac{\partial x}{\partial s}_{|(s,t)}, \ \ x(s,0)=x(s).$$ 
Notice that, although $x(s)$ is contained in $V_i$, we should not automatically assume that the values of $x(s+it)$, for $|s+it|<\epsilon$, are contained in $P_i$. However, our method of bounding $|x(s)|$ applies also to $|x(s+it)|$: this follows from the general fact that one can bound $|\sum a_n(s+it)^n|$ with $\sum |a_n| |s+it|^n$.

As explained in Step 2, this concludes the proof of Theorem \ref{thm:exp}.

\subsection{Example: the 1-dimensional case}\label{ss:dim1}

We now turn to smooth data. It is instructive to study the Cauchy problem in the simplest case, where $M=\C$ and $L_0$ is a smooth, closed, Jordan curve. Recall that $\C\setminus L_0$ has two components: one bounded, one unbounded. We   view $L_0$ as an embedding $\iota_0=\iota_0(\theta)$ of the abstract manifold $L:=\R/2\pi\Z$. For dimensional reasons any such embedding is totally real. To be concrete, we use the standard orientation on $L$ defined by increasing angles and we assume the embedding is chosen so that $L_0$ is oriented in the anti-clockwise direction.
 
By our definition, geodesics through $L_0$  are generated by a choice of tangent vector field $X$. Since $L$ is parallelizable and has the canonical, positively oriented, vector field $\partial\theta$, we have $X=f\partial\theta$, for some $f: L\rightarrow \R$. The corresponding geodesic in $\mathcal{T}$ is determined by the 1-parameter family of curves
$$\iota:L\times (-\epsilon,\epsilon) \rightarrow \C$$
such that $\iota=\iota_0$ for $t=0$ and
\begin{equation}\label{eq:curve_geodesic}
 \frac{\partial\iota}{\partial t}=if\frac{\partial\iota}{\partial \theta}.
\end{equation}
This coincides with the geodesic equation \eqref{geod.eq}. In particular, when $f\equiv 1$ this means that $\iota$ is holomorphic with respect to the standard complex structure on the cylinder $L\times(-\epsilon,\epsilon)$. We can use the biholomorphism with the annulus
\begin{equation}\label{eq:bihol_annulus}
\phi:L\times(-\epsilon, \epsilon)\rightarrow A:=\{e^{-\epsilon}<|z|<e^{\epsilon}\}, \ \ \phi(\theta,t):=e^{-t}e^{i\theta}
\end{equation}
to reparametrize $\iota$ as a holomorphic map $g:=\iota\circ\phi^{-1}:A\rightarrow\C$. Our choice of orientations imply that, as $t$ increases from $0$, the geodesics invade the bounded component of  $\C\setminus L_0$.

We now show that the $f\equiv 1$ case is, in some sense,  general. 
Indeed, using the ideas of Section \ref{ss:reformulation}, we can  integrate $X=f\partial\theta$. If $f$ has no zeros, \textit{i.e.}~$X$ never vanishes, then the integral curve through any    point of $L$ is periodic and its parameter set is compact: we can identify it with $\S^1_R:=\R/2\pi R\Z$, for some $R>0$. The integral curve is then a (possibly orientation-reversing) diffeomorphism 
\begin{equation}
 \S^1_R\rightarrow L, \ \ s\mapsto \theta(s)\ \ \mbox{such that}\ \ \theta'=f\circ\theta. 
\end{equation}
The map $\iota(\theta(s),t)$ is holomorphic on the cylinder $\S^1_R\times(-\epsilon,\epsilon)$ (with the standard complex structure). Again, we can use the biholomorphism with the annulus
\begin{equation}\label{eq:bihol_annulus_R}
\phi:\mathcal{S}^1_R\times(-\epsilon, \epsilon)\rightarrow A_R:=\{e^{-\epsilon/R}<|z|<e^{\epsilon/R}\}, \ \ \phi(\theta,t):=e^{-t/R}e^{i\theta/R}
\end{equation}
to reparametrize $\iota$ as a holomorphic map $g:A_R\rightarrow\C$. Notice that this implies a rescaling of the initial vector field $X$.

We summarize this discussion as follows.

\begin{lem}\label{l:geodesic_annulus}
Fix $L_0$ and a nowhere-vanishing vector field $X$ as above. Then, up to a rescaling of $X$, the geodesic family of curves $L_t$ defined by this data is equivalent to a holomorphic map $g$ defined on an annulus in $\C$ containing $\Sph^1$. Each $L_t$ is the image under $g$ of some circle $\{|z|=r\}$; in particular  $L_0$ is the image of\/ $\Sph^1$ and $X$ corresponds to $\pm\partial\theta$ depending on the sign of $f$. 

Assume, for example, $f>0$. Then, as the radial parameter $r$ decreases from $1$, the corresponding curves invade the bounded component  of $\C\setminus L_0$.
\end{lem}

\begin{remark}
Our discussion above indicates that the 1-dimensional case has a special feature. Recall from Lemma \ref{l:H_defines_connection} that the horizontal distribution $H$ on $\mathcal{P}$ is invariant under reparametrization. This means  we can find all geodesics through $L_0$ by fixing any initial parametrization $\iota_0$ and considering all possible vector fields: the geodesic in $\TRL$ defined by a different choice $(\iota_0 \circ\phi,X)$ will coincide with the geodesic defined by $(\iota_0,\phi_*X)$. Thus, in general there is no advantage to changing the parametrization. In dimension 1, however, $\Diff(L)$ acts transitively on the non-vanishing vector fields (up to a change of scale). Above, we use this to  bring $X$ into ``standard form'' $\partial\theta$, thus reducing the geodesic equation \eqref{geod.eq} to the standard Cauchy--Riemann equation. However, note that, when $f<0$, this strategy clashes with our initial decision to work with oriented submanifolds,
 \textit{i.e.}~to only 
use orientation-preserving diffeomorphisms: this is easily fixed by the observation that the geodesics defined by $X$ and $-X$ coincide, up to time reversal. To find all geodesics, it it thus enough to concentrate on those for which $f>0$. A similar remark applies to vector fields with zeros (see below).
\end{remark}

If the vector field $X=f\partial\theta$ has zeros, then between any two zeros
 the new parameter set will be $\R$ and the geodesic equation \eqref{geod.eq} pulls back
 to the standard Cauchy--Riemann equation on $\R\times (\epsilon,\epsilon)$. The
  zeros  correspond to stationary points of the family of curves.

As already mentioned, there is a necessary condition for the existence of solutions to this equation: the initial curve must be real analytic. This condition is also sufficient: given a local power series expansion of $\iota_0$ with respect to the real variable $\theta$, we obtain a holomorphic extension by replacing $\theta$ with $\theta+it$. 

In Section \ref{ss:reformulation}, in order to remain in the smooth category, we introduced geodesic rays. Using the above ideas, we can study geodesic rays in the 1-dimensional case and obtain a conclusion analogous to Lemma \ref{l:geodesic_annulus}. In particular, the geodesic ray defined by $L_0$ and a non-vanishing vector field $X$ is equivalent to a holomorphic map $g$ defined on an annulus in $\C$ of the form $R_1<|z|<1$, smooth up to $|z|=1$. Elliptic regularity theory shows that, if the boundary data is smooth, then $g$ is smooth up to $|z|=1$ even if in Definition \ref{def:geodesic_ray} we assumed the geodesic ray were only continuous with respect to $t$, at $t=0$.

These results allow us to study the existence of geodesics using holomorphic function theory.

\paragraph{Geodesics via Fourier theory.}
We showed above that an initial curve $L_0$ and non-vanishing vector field $X$ can be parametrized (up to rescaling $X$) via a map $\gamma:\Sph^1\rightarrow\C$ and the standard vector field $\partial\theta$. By Lemma \ref{l:geodesic_annulus}, this data defines a geodesic if and only if it admits a holomorphic extension $g$ on an annulus $A$. Recall from standard theory that such $g$ admit a Laurent power series representation  $\sum_{n=-\infty}^{\infty}a_nz^n$, convergent on $A$. The coefficients $a_n$ coincide with the Fourier coefficients of the periodic function $\gamma$. It follows that the existence of $g$, \textit{i.e.} of the geodesic, depends on the convergence of the formal power series defined by the Fourier coefficients of $\gamma$.

Recall that the Fourier coefficients of the curve are square-summable. Conversely, given a square-summable sequence of  coefficients $a_n\in\C$, for $n\in\Z$, we can ask whether it defines a curve $\gamma:\Sph^1\rightarrow\C$ admitting a holomorphic extension $g$. Let $p_1(z)=\sum_{n=-1}^{-\infty}a_nz^n$ and $p_2(z)=\sum_{n=0}^{\infty}a_nz^n$.

\begin{itemize}
 \item If $p_1$ and $p_2$ have radii of convergence $R_1<1$ and $R_2>1$ respectively, then the Laurent series $p_1+p_2$ converges on the annulus $R_1<|z|<R_2$ and defines an embedding $\gamma$ of $\Sph^1$. The image curve $L_0$ admits a geodesic corresponding to $\partial\theta$.\vspace{-4pt}
 \item If $p_1$ has radius of convergence $R_1<1$ and $p_2$ has radius of convergence $1$ and converges for $|z|=1$, then the Laurent series $p_1+p_2$ converges on $R_1<|z|\leq 1$ and defines an embedding $\gamma$ of $\Sph^1$. The image curve $L_0$ admits a geodesic ray corresponding to  $\partial\theta$.\vspace{-4pt}
 \item If $p_1$ and $p_2$ have radius of convergence $1$ and converge for $|z|=1$ then the Laurent series degenerates: it converges only on $\Sph^1$, defining a  curve $L_0$ which does not admit a geodesic or geodesic ray corresponding to  $\partial\theta$.
\end{itemize}

\begin{example}
 Set $a_n:=1/n^{\log n}$ for $n\geq 1$. Then  $\sum_{n=1}^\infty a_n z^n$ has radius of convergence $1$ and converges absolutely for $|z|\leq 1$, together with all derivatives. Adding this to any series $\sum_{n=-1}^{-\infty} a_nz^n$ with radius of convergence $R_1<1$ gives smooth curves which admit geodesic rays but not geodesics corresponding to $\partial\theta$. We can also combine it with $\sum_{n=-1}^{-\infty}|n|^{-\log |n|}z^n$ to obtain a smooth curve which admits neither a geodesic nor a geodesic ray corresponding to that $\partial\theta$.

 To obtain examples which are only continuous, set $a_n:=1/n^2$.
 \end{example}

\paragraph{Geodesics via the Riemann Mapping Theorem.}Choose two closed Jordan curves $L_0$, $L_1$ in $\C$ which do not intersect. 
Let $\Omega$ be the region contained between these curves. A version of the Riemann mapping theorem, cf.~\cite[Theorem 5.8]{Conway}, proves that there exists an annulus $A$ and a biholomorphism $g:A\rightarrow\Omega$ which extends continuously to the boundary; if $L_0$, $L_1$ are smooth then the biholomorphism extends smoothly to the boundary. The restriction to the boundary provides parametrizations of $L_0$, $L_1$; setting $X=\partial\theta$ the theorem shows that for any two curves as above it is possible to solve the boundary value problem.

\begin{remark}
 Notice the regularizing behaviour of the geodesic equation \eqref{geod.eq} even for the boundary value problem: for all intermediate times $t\in (0,1)$, the corresponding curves are real analytic.
\end{remark}

\paragraph{Concluding remarks.} We summarize what we have learned from the 1-dimensional theory. Given an embedded curve $L_0\subset\C$, we  showed  the following.
\begin{itemize}
\item Infinitesimal deformations correspond to parametrizations (via integration of the tangential vector field $f\partial \theta$).\vspace{-4pt}
\item A geodesic in a given direction corresponds to a holomorphic extension of the corresponding parametrization.\vspace{-4pt}
\item Examples show certain curves do not admit geodesics in certain directions.\vspace{-4pt}
\item Given any curve $L_0$, there always exist infinite geodesic rays departing from it (corresponding to the arbitrary choice of a  second curve $L_1$).
\end{itemize}
This suggests the existence question for geodesics is non-trivial, but  not vacuous. A similar situation occurs in the analogous theory concerning K\"ahler metrics, cf.~Section \ref{s:pottyK}. There a weak notion of geodesics was found, leading to a satisfactory existence theory. We expect something similar is needed here. In particular, observe that our geodesic equation \eqref{geod.eq} is first order, rather than second order as one might except: this corresponds to the fact that, in keeping with the principal fibre bundle viewpoint, it is expressed in terms of the velocity vector (being constant) rather than of the curve. Developing alternative expressions for geodesics and further investigation of the properties of the connection may contribute key ingredients to the existence theory.

\subsection{Further results}

Some of these same ideas can be extended to higher dimensions. 

\paragraph{Existence and non-existence results when {\boldmath $M=\C^n$}.} Consider a compact totally real submanifold $L_0$ in $\C^n$ and a tangent vector field $X$. Choose an integral curve $x=x(s)$ and a parametrization $\iota_0$, with components $\iota_0^i$. If the curve $x$ is closed we can study the existence of holomorphic extensions of $\gamma:=\iota_0\circ x(s)$ exactly as when $n=1$, by examining its component functions $\iota_0^i\circ x(s)$. This does not work if the curve is open, parametrized by $\R$. Notice however that the image of $\gamma$ is contained in $L_0$, so it is bounded. We can thus interpret $\gamma$ as a (smooth) tempered distribution and replace the role of Fourier coefficients with Fourier transforms. In particular, we expect to obtain information concerning  existence of holomorphic extensions of $\gamma$ using the Paley--Wiener theorems. It is known for example, cf.~\cite[Theorem 7.23]{Rudin}, that if the transform of $\gamma$ has compact support then $\gamma$ admits an 
entire 
holomorphic extension (satisfying certain growth conditions). Notice that in this case the transform of $\gamma$ will generally not be smooth, otherwise it would be $L^2$ so $\gamma$ would also be $L^2$.

This applies also to any complex manifold $M$, as long as the submanifold is contained in one chart.

\paragraph{Uniqueness of geodesics.}
Perhaps the most interesting feature of our reformulation of the geodesic equation \eqref{geod.eq} is that it gives a fairly complete answer to the uniqueness question. Indeed, by restricting $\iota_0$ to each integral curve we see that it suffices to prove the following claim: any two $J$-holomorphic maps $\iota(s,t)$, $\iota'(s,t)$ which coincide for $t=0$ coincide for all $t$.

In the holomorphic case (when $J$ is integrable) the proof is simple. As above, $\iota$ corresponds locally to a collection of holomorphic functions, defined by its components in $\C^n$. Uniqueness for the Cauchy problem then follows from the standard identity principle for holomorphic functions. Uniqueness for geodesic rays follows instead from the standard reflection principle.

If $J$ is only almost complex the situation is more subtle. Uniqueness for the Cauchy problem is then a consequence of the ``unique continuation theorem'' for $J$-holomorphic curves, cf.~\cite[Theorem 2.3.2]{MS}. It seems reasonable that, using results in the literature, one could also prove uniqueness for geodesic rays.

\begin{remark}
In the real analytic case the uniqueness of real analytic solutions is  part of the Cauchy--Kowalevski theorem. One might hope to improve on this, obtaining uniqueness in the smooth category, using Holmgren's uniqueness theorem, cf.~\cite[Chapter 21]{Treves}.
However, Holmgren's theorem concerns only linear equations and this corresponds to an important difference between the holomorphic and  pseudo-holomorphic equations. In the former case, in local coordinates, the operator $J$ is constant so the Cauchy--Riemann equation is indeed linear. Holmgren's theorem thus applies to give an alternative proof of the uniqueness of geodesics and geodesic rays.
In general almost complex manifolds, instead, the Cauchy--Riemann equation is not locally linear.
\end{remark}

\subsection{Example: the 1-dimensional case, continued}
We now examine the notion of geodesic convexity from Definition \ref{convex.dfn} in the 1-dimensional case by exhibiting an example of a convex functional. This functional is well-known: it is the standard length functional. Its convexity is a rather striking fact, and it is  worth emphasizing it by giving two proofs.  The first relies on the specific nature of the geodesic equation by bringing into play  
basic holomorphic function theory. As above, it assumes we have reparametrized the curve by integrating $f\partial\theta$, but it requires that the domain 
remains compact. This first proof also leads to a monotonicity result for the length functional. The second proof is a direct geometric calculation, and
 holds for all $f$.

\begin{prop}
The length functional on closed curves in $\C$ is convex in the sense of Definition \ref{convex.dfn}.
\end{prop}

\begin{proof}
For the first proof, assume we are given a holomorphic map on the cylinder 
$$\gamma:\S^1_R\times(-\epsilon,\epsilon) \rightarrow \C,$$
where $\S^1_R=\R/2\pi R\Z$.  
Let $w=s+it$ denote the complex variable on the cylinder and $\lambda=\lambda(t):(-\epsilon, \epsilon)\rightarrow\R$ the length of the curve $\gamma(\cdot, t)$. Explicitly,
$$\lambda=\int_0^{2\pi R}\left|\frac{\partial \gamma}{\partial s}\right|\,ds=\int_0^{2\pi R}\left|\frac{\partial \gamma}{\partial w}\right|\,ds.$$
We want to prove that $\lambda$ is convex with respect to $t$.

The biholomorphism $z=\phi(s,t)$ in \eqref{eq:bihol_annulus_R}
allows us to reformulate the problem in terms of a map $g=g(z):A_R\rightarrow\C$ such that $\gamma=g\circ\phi$. Notice that $g$ is holomorphic if and only if $\gamma$ is holomorphic and their complex derivatives satisfy $|\frac{\partial\gamma}{\partial w}|=(1/R)|\phi\frac{\partial g}{\partial z}|$. Using polar coordinates on $\C$ and setting 
$$\Lambda(r):=\int_0^{2\pi}\left|z\frac{\partial g}{\partial z}\right|\,d\theta,$$
it follows that $\lambda(t)=\Lambda(e^{-t/R})$. It thus suffices to prove that $\Lambda\circ \exp$ is convex, \textit{i.e.} that on any segment $[t_1,t_2]$ the graph of $t\mapsto\Lambda(e^t)$ is below the graph of the linear function passing through the points $(t_1,\Lambda(e^{t_1}))$, $(t_2,\Lambda(e^{t_2}))$.

Notice that  
$z\frac{\partial g}{\partial z}$ is holomorphic, so its norm $u:=|z\frac{\partial g}{\partial z}|$ is a subharmonic function on the 
annulus. Convexity is then a classical result due to Riesz and proved as follows. Let $v$ denote the harmonic function on the annulus $A:=\{z:r_1\leq |z|\leq r_2\}\subseteq A_R$ such that $v$ coincides with $u$ on the boundary. Notice that
$$\frac{d}{dr}\int_0^{2\pi} v(re^{i\theta})\,d\theta=\int_0^{2\pi} \frac{d}{dr} v(re^{i\theta})\,d\theta=\frac{1}{r}\int_0^{2\pi} \frac{\partial v}{\partial n}(re^{i\theta})\,d\sigma,$$
where $d\sigma=r\,d\theta$. Since $v$ is harmonic, the divergence theorem shows that $r\mapsto\int_0^{2\pi} \frac{\partial v}{\partial n}\,d\sigma$ is constant, so by subharmonicity  
$$\int_0^{2\pi} u(re^{i\theta})\,d\theta\leq \int_0^{2\pi} v(re^{i\theta})\,d\theta=a\log r+b,$$
for some constants $a,b\in\R$; our choice of boundary data implies that equality holds when $r=r_1$ or $r=r_2$. Changing variables we obtain the desired property of $\Lambda(e^t)$.

Similar methods show that if $g$ extends to a holomorphic function on the disk then $\Lambda(r)$ is non-decreasing, so $\lambda(t)$ is non-increasing.   

\ 

For the second proof, we first parametrize the curve  by arclength: it is thus the image of some map $\gamma_0(s)$, where $s\in L$. Choose a vector field $X=f\partial s$ and let $\gamma(s,t)=\gamma_t(s)$ be a family of curves satisfying the corresponding geodesic equation (\ref{eq:curve_geodesic}). Set $\gamma':=\frac{\partial\gamma}{\partial s}$ and $\dot{\gamma}:=\frac{\partial\gamma}{\partial t}$ for a cleaner exposition.
 
The length functional along $\gamma_t$ is given by 
$$\lambda(\gamma_t)=\int_L|\gamma_t'|\d s.$$

We first calculate
\begin{align*}
\frac{\d}{\d t}\lambda(\gamma_t)&=\int_L\frac{\partial}{\partial t}\langle \gamma_t' , \gamma_t'\rangle^{\frac{1}{2}}\d s=\int_L|\gamma_t'|^{-1}
\langle \dot{\gamma_t}', \gamma_t'\rangle\d s.
\end{align*}
Then
\begin{align*}
\frac{\d^2}{\d t^2}\lambda(\gamma_t)&=\int_L|\gamma_t'|^{-1}\left(\langle \ddot{\gamma_t}',\gamma_t'\rangle+|\dot{\gamma_t}'|^2\right)
-|\gamma_t'|^{-3}\langle \dot{\gamma_t}',\gamma_t'\rangle^2\d s.
\end{align*}
Since $|\gamma_t'|^{-1}\gamma_t'$ is a unit tangent vector we have that 
\begin{align*}
\frac{\partial}{\partial s}\left(|\gamma_t'|^{-1}\gamma_t'\right)&=|\gamma_t'|^{-1}\gamma_t''-|\gamma_t'|^{-3}\langle \gamma_t'', \gamma_t'\rangle\gamma_t'
=i\kappa_t|\gamma_t'|^{-1}\gamma_t',
\end{align*}
where $\kappa_t$ is the curvature of $\gamma_t$.  Therefore,
$$\gamma_t''=|\gamma_t'|^{-2}\langle\gamma_t'',\gamma_t'\rangle\gamma_t'+i\kappa_t\gamma_t'.$$
Hence,
$$\dot{\gamma_t}'=\frac{\partial}{\partial s}\dot{\gamma_t}=\frac{\partial}{\partial s}(if\gamma_t')
=if\gamma_t''+if'\gamma_t'=-f\kappa_t\gamma_t'+i(f'+f|\gamma_t'|^{-2}\langle\gamma_t'',\gamma_t'\rangle)\gamma_t'.$$
Moreover,
$$\ddot{\gamma_t}=\frac{\partial}{\partial t}(if\gamma_t')=if\dot{\gamma_t}'=
-f(f'+f|\gamma_t'|^{-2}\langle\gamma_t'',\gamma_t'\rangle)\gamma_t'-if^2\kappa_t\gamma_t'.$$

Since we will be taking no further $t$ derivatives and $\gamma_0$ was arbitrary we can now set $t=0$ without loss of generality. 
In this case, 
because $|\gamma_0'|=1$ we see that $\langle\gamma_0'',\gamma_0'\rangle=0$ and thus $\gamma_0''=i\kappa\gamma_0'$ where $\kappa_0$ is the curvature of $\gamma_0$.  Hence, 
$$\dot{\gamma_t}'|_{t=0}=-f\kappa_0\gamma_0'+if'\gamma_0'\quad\text{and}\quad \ddot{\gamma_t}|_{t=0}=-ff'\gamma_0'-if^2\kappa_0\gamma_0'.$$
We therefore see that 
\begin{align*}
\ddot{\gamma_t}'|_{t=0}&=-(ff')'\gamma_0'-ff'\gamma_0''-i(f^2\kappa_0)'\gamma_0'-if^2\kappa_0\gamma_0''\\
&=(f^2\kappa_0^2-(ff')')\gamma_0'-i(ff'\kappa_0+(f^2\kappa_0)')\gamma_0'.
\end{align*}

Putting these formulae together we see that
\begin{align*}
\langle\ddot{\gamma_t}'',\gamma_t'\rangle|_{t=0}&=f^2\kappa_0^2-(ff')', \quad
|\dot{\gamma_t}'|^2|_{t=0}=f^2\kappa_0^2+(f')^2, \quad
\langle\dot{\gamma_t}',\gamma_t'\rangle^2=f^2\kappa_0^2.
\end{align*}
We deduce that
\begin{align*}
\frac{\d^2}{\d t^2}\lambda(\gamma_t)|_{t=0}&=\int_Lf^2\kappa_0^2-(ff')'+f^2\kappa_0^2+(f')^2-f^2\kappa_0^2\,\d s\\
&=\int_L(f')^2+f^2\kappa_0^2\,\d s\geq 0
\end{align*}
since $\int_L(ff')'\d s=0$.  Therefore the length $\lambda(\gamma_t)$ is a convex function of $t$.
\end{proof}

\section{A canonical volume functional}\label{s:canonical_data}
In higher dimensions the standard Riemannian volume functional is not convex with respect to our notion of geodesics. This is hardly surprising: when $n\geq 2$ the totally real condition is an extra assumption, but the volume functional does not interact with this condition. The goal of this section is to show that, for totally reals, there is an alternative volume functional which (i) is canonical, (ii) depends on the totally real condition and (iii) is convex in certain situations. 

To define this functional we need an alternative characterization of totally real planes: an $n$-plane $\pi$ in $T_pM$ is totally real if and only if $\alpha|_{\pi}\neq 0$ 
for all (or, for any) $\alpha\in K_M(p)\setminus\{0\}$, where $K_M$ is the canonical bundle of $(M,J)$.  

Notice that $n$-planes $\pi$ in $T_pM$ which are not totally real contain a complex line:  a pair $\{X,JX\}$ for some $X\in T_pM\setminus\{0\}$.  We 
call such $n$-planes \emph{partially complex}.  We then say that an $n$-dimensional submanifold is \textit{partially complex} if this condition holds in the strongest sense, \textit{i.e.}~if each of its tangent spaces is partially complex.

Let $TR^+$ denote the Grassmannian bundle of oriented totally real $n$-planes in $TM$ and let 
$\pi\in TR^+(p)$. Let $v_1,\ldots,v_n$ be a positively oriented basis of $\pi$. We can then define $v_j^*\in T_p^*M\otimes\C$ by
$$v_j^*(v_k)=\delta_{jk}\quad\text{and}\quad v_j^*(Jv_k)=i\delta_{jk}.$$
This allows us to define a non-zero form $v_1^*\w\ldots\w v_n^*\in K_M(p)$.

The form we have constructed depends on the choice of basis $v_1,\ldots,v_n$. 
We fix this by assuming  we have a Hermitian metric $h$ on $K_M$, and  then define
$$\sigma[\pi]=\frac{v_1^*\w\ldots\w v_n^*}{|v_1^*\w\ldots\w v_n^*|_h}\in K_M(p).$$
This form has unit norm and is now independent of the choice of basis.

We have thus defined a map between bundles\footnote{In fact, $\sigma$ maps into the unit circle bundle in 
$K_M$.} $\sigma:TR^+\rightarrow K_M$ covering the identity.  We also see that 
the restriction of $\sigma[\pi]$ to $\pi$ is a \emph{real-valued} $n$-form.

Now let $\iota:L\rightarrow M$ be an $n$-dimensional totally real immersion. We can then obtain global versions of the above constructions as follows. 
\paragraph{Canonical bundle over {\boldmath $L$}.}Let $K_M[\iota]$ denote the pullback of $K_M$ over $L$, so the fibre of $K_M[\iota]$ over $p\in L$ is the fibre of $K_M$ over $\iota(p)\in M$. This defines a complex line bundle over $L$ which depends on $\iota$. 

Observe that any complex-valued $n$-form $\alpha$ on $T_pL$ defines a unique $n$-form $\widetilde{\alpha}$ 
on $T_{\iota(p)}M$ by identifying $T_pL$ with its image via $\iota_*$ and setting, \textit{e.g.}, 
$$\widetilde{\alpha}[\iota(p)](J\iota_*(v_1),\dots,J\iota_*(v_n)):=i^n\alpha[p](v_1,\dots,v_n).$$ 
The totally real condition implies that this is an isomorphism:  $K_M[\iota]$ is canonically isomorphic, via $\iota_*$, with the ($\iota$-independent) bundle $\Lambda^n(L,\C):=\Lambda^n(L,\R)\otimes\C$ of complex-valued $n$-forms on $L$.

\paragraph{Canonical section.}Now we use the fact that $L$ is  \emph{oriented}. Hence, $\Lambda^n(L,\R)$ is trivial, so $K_M[\iota]$ also is. We build a global section of $K_M[\iota]$ using our previous linear-algebraic construction: $p\mapsto\Omega_J[\iota](p):=\sigma[\iota_*(T_pL)]$. We call $\Omega_J[\iota]$ the \textit{canonical section} of $K_M[\iota]$. Restricting $\Omega_J[\iota]$ to $\iota_*(T_pL)$ yields a \emph{real-valued} positive $n$-form on $\iota(L)$, thus a volume form $\vol_J[\iota]:=\iota^*(\Omega_J[\iota])$ on $L$: we call it the \textit{$J$-volume form} of $L$, defined by $\iota$.  

When $L$ is compact we obtain a ``canonical volume'' $\int_L \vol_J[\iota]$, for $\iota\in\mathcal{P}$.  One may  see that if $\varphi\in\Diff(L)$ then $\vol_J[\iota\circ\varphi]=\varphi^*(\vol_J[\iota])$, just as for the standard volume form, thus 
$$\int_L\vol_J[\iota\circ\varphi]=\int_L\varphi^*\vol_J[\iota]=\int_L\vol_J[\iota].$$  Hence the canonical volume descends to $\mathcal{T}$  to define the \textit{$J$-volume functional}

\begin{equation*}
 \Vol_J:\TRL\rightarrow \R,\ \ L\mapsto\int_L\vol_J[\iota],
\end{equation*}
where $\iota$ is any parametrization representing the submanifold $L$.

Already in this context it would be possible to study its first variation, thus characterizing its critical points. Using the connection on $\TRL$ one could also define its second variation, studying the stability properties of the critical points. We will do this below, in the presence of additional structure and hypotheses on $M$ which will allow us to determine a useful expression for the first variation and a simplified formula for the second variation.

\textit{Notation.} From now on we will sometimes simplify notation by dropping the reference to the specific immersion used. 
Since this is standard practice in other contexts, \textit{e.g.}~when discussing the standard Riemannian volume, we expect it will not create any confusion.

\section{The {\boldmath $J$}-volume in the Hermitian context}\label{s:Jvol}

Assume now that $(M,J)$ is almost Hermitian, \textit{i.e.}~we have a Riemannian metric $\overline{g}$ on $M$ compatible with $J$, so 
$J$ is an isometry defining a Hermitian metric $h$ on $M$. We also choose a unitary connection $\tnabla$ on $M$. 

Let $L$ be an oriented totally real submanifold of $(M,J)$. In Riemannian geometry it is customary to work with tangential and normal projections $\pi_{\rm T}$, $\pi_{\perp}$ and the Levi-Civita connection $\overline{\nabla}$. This however does not make  use of the totally real condition which implies that, for any $p\in L$, any vector $Z\in T_pM$ can  be written uniquely as $Z=X+JY$ where $X,Y\in T_pL$. This splitting induces projections $\pi_L,\pi_J$ by
 setting $\pi_L(Z)=X$ and $\pi_J(Z)=JY$: these are the natural projections in this context. The following fact is a simple computation.
 
 \begin{lem}\label{proj.lem}
 $\pi_L\circ J=J\circ\pi_J$ and $J\circ\pi_L=\pi_J\circ J$.
\end{lem}

The structures on $M$ induce structures $h$, $\tnabla$ on $K_M$, which we can use to define the $J$-volume form on $L$. 

Notice that, in contrast to Section \ref{s:canonical_data} where we  only had a complex structure, we now have the 2-form 
$\overline{\omega}(\cdot,\cdot)=\overline{g}(J\cdot,\cdot)$. In this section we can thus also discuss Lagrangian submanifolds, defined by the condition $\iota^*\bar{\omega}=0$.

\subsection{The {\boldmath $J$}-volume versus the Riemannian volume}\label{ss:vol_comparison}

In the almost Hermitian context, given an immersion $\iota$, we can define the usual Riemannian volume form $\vol_g$, using the induced metric $g$. It is useful to compare this with the $J$-volume form, cf.~also \cite{LP}.

Let $e_1,\dots,e_n$ be a positive orthonormal basis  for $\pi$ and set $h_{ij}=h(e_i,e_j)$, where $h$ is the ambient Hermitian metric.  We wish to calculate $|e_1^*\w\ldots\w e_n^*|_h$.  
Observe that  $h(.,e_j)=h_{kj}e_k^*$ since $$h(e_i,e_j)=h_{ij}=h_{kj}e_k^*(e_i)$$ and $$h(Je_i,e_j)=ih_{ij}=ih_{kj}e_k^*(e_i)=h_{kj}e_k^*(Je_i).$$
Thus 
$$h(.,e_1)\w\ldots\w h(.,e_n)=({\det}_{\C}h_{ij}) e_1^*\w\ldots\w e_n^*.$$
Hence
$$|e_1^*\w\ldots\w e_n^*|_h=({\det}_{\C}h_{ij})^{-1}|h(\cdot,e_1)\w\ldots\w h(\cdot,e_n)|_h.$$
We now notice that 
$$|h(\cdot,e_1)\w\ldots\w h(\cdot,e_n)|_h^2={\det}_{\C}h_{ij}$$
so
\begin{equation}\label{h.mod.eq}
|e_1^*\w\ldots\w e_n^*|_h=({\det}_{\C}h_{ij})^{-\frac{1}{2}}.
\end{equation}
We therefore find that
\begin{align*}
 \vol_J&=\frac{e_1^*\w\ldots\w e_n^*}{|e_1^*\w\ldots\w e_n^*|_h}{|_\pi}=({\det}_{\C}h_{ij})^{1/2}\vol_g.
\end{align*}
We can now obtain a well-defined function 
\begin{equation*}\label{rhoJ.eq.1}
\rho_J:TR^+\rightarrow \R,\ \ \rho_J(\pi):=\vol_J(e_1,\dots,e_n)=({\det}_{\C}h_{ij})^{1/2},
\end{equation*}
because this quantity is independent of the orthonormal basis $e_1,\ldots,e_n$. 

Restricting $\rho_J$  to an oriented totally real submanifold $L$, 
we obtain the identity: $\vol_J=\rho_J \vol_g$.  

Notice that $h=\overline{g}-i\overline{\omega}$ and that, using the obvious notation for the components of $\overline{g}$ and $\overline{\omega}$ with respect to $e_1,\ldots,e_n$,
\begin{align*}
{\det}_{\C}h_{ij}&=\sqrt{\det\left(\begin{array}{cc}\overline{g}_{ij} & \overline{\omega}_{ij}\\ -\overline{\omega}_{ij} & \overline{g}_{ij}\end{array}\right)}.
\end{align*}
We see that 
$$\overline{\omega}_{ij}=\overline{g}(Je_i,e_j)\quad\text{and}\quad-\overline{\omega}_{ij}=\overline{g}(e_i,Je_j).$$  We deduce that $\det_{\C}h_{ij}=\sqrt{\det(\overline{g}_{ab})}$ where $\overline{g}_{ab}$ is the 
matrix of $\overline{g}$ with respect to the basis $\{e_1,\ldots,e_n,Je_1,\ldots,Je_n\}$.  Therefore
$${\det}_{\C}h_{ij}=\vol_{\overline{g}}(e_1,\ldots,e_n,Je_1,\ldots,Je_n).$$

We deduce a second expression for $\rho_J$: 
\begin{equation}\label{rhoJ.eq.2}
\rho_J(\pi)=\sqrt{\vol_{\overline{g}} (e_1,\dots,e_n,Je_1,\dots,Je_n)}.
\end{equation}
Hence we see that $\rho_J(\pi)\leq 1$ with equality if and only if $\pi$ is Lagrangian. 

More generally, given any basis $v_1,\dots,v_n$ for $\pi$, we can write
\begin{equation}\label{rhoJ.eq.3}
\rho_J(\pi)=\frac{\sqrt{\vol_{\overline{g}} (v_1,\dots,v_n,Jv_1,\dots,Jv_n)}}{|v_1\wedge\dots\wedge v_n|_{\bar g}}.
\end{equation}

We can set $\rho_J(\pi)=0$ when $\pi$ is partially complex and extend the map $\sigma$ to all $n$-planes, just setting $\sigma[\pi]=0$ if $\pi$ is partially complex. 
This is particularly reasonable in this almost Hermitian setting, where there is a natural topology on the Grassmannian of $n$-planes: this choice of extension of $\sigma$ would be justified by the fact that it is the unique one which preserves the continuity of $\sigma$.

Applying these observation to submanifolds we deduce the following.

\begin{lem}\label{Lag.Jvol.lem}
For any compact oriented $n$-dimensional submanifold $L$ in an almost Hermitian
manifold $(M,J,\overline{g})$, we have $\Vol_J(L)\leq \Vol_g(L)$ with equality if and only if $L$ is Lagrangian. In particular, $\Vol_J$ and $\Vol_g$ coincide to first order on Lagrangians.
\end{lem}

\begin{proof}
The first statement follows from (\ref{rhoJ.eq.2}). To prove the second, let $L_t$ be a 1-parameter family of totally real submanifolds such that $L_0$ is Lagrangian. Set $f(t):=\Vol_J(L_t)$ and $g(t):=\Vol_g(L_t)$. Then $f\leq g$ so $g-f\geq 0$. Equality holds when $t=0$: this is a minimum point, so it is necessarily critical. It follows that $f'(0)=g'(0)$. Since this holds for any 1-parameter family, we obtain the desired conclusion.  
\end{proof}

\subsection{First variation of the {\boldmath $J$}-volume}

\begin{prop}\label{volJ.prop}
Let $\iota_t:L\rightarrow L_t\subseteq M$ be a one-parameter family of totally real submanifolds in an almost Hermitian manifold and let 
$\frac{\partial}{\partial t}\iota_t|_{t=0}=Z$. 

Set $\iota=\iota_0$ and $g=\iota^*\bar g$. If $Z=X+JY$ for tangent vectors $X,Y$ then 
\begin{align*}
\frac{\partial}{\partial t}\vol_J[\iota_t]|_{t=0}&=\overline{g}\big(\pi_L\big(\tnabla_{e_i}Z+\widetilde{T}(Z,e_i)\big),e_i\big)\vol_J[\iota]\\
&=\Div(\rho_JX)\vol_g
+\overline{g}\big(\pi_L\big(\tnabla_{e_i}JY+\widetilde{T}(JY,e_i)\big),e_i\big)\vol_J[\iota],
\end{align*}
where at $p\in L$ we have that $e_1,\ldots,e_n$ is a $g$-orthonormal basis for $T_pL$.
\end{prop}

\noindent Notice that the quantities which appear in the above formulae are invariant under changes of orthonormal basis and so are globally defined.

\begin{proof} Let $p\in L$, let $x_1,\ldots,x_n$ be normal coordinates at $p$ with respect to $g$ and let $e_i=\frac{\partial}{\partial x_i}$: this defines an orthonormal basis for $T_pL$.
Since we have $L_t=\iota_t(L)=\iota(L,t)$ where $\iota(x,t):=\iota_t(x)$, we may consider coordinates $(x_1,\ldots,x_n,t)$ near $(p,0)$ in $L\times(-\epsilon,\epsilon)$, 
so then $\iota_*(\frac{\partial}{\partial t})=Z$ at $t=0$.
Notice that since $(x_1,\ldots,x_n,t)$ is a coordinate system on $L\times(-\epsilon,\epsilon)$ we see that $[Z,e_i]=0$. 

Recall that the $J$-volume form is given at $p$ by
$$\vol_J[\iota]=\rho_J(\pi)\vol_g=\sqrt{\vol_{\overline{g}}(e_1,\ldots,e_n,Je_1,\ldots,Je_n)}\vol_g,$$
where $\pi=T_{\iota(p)}L$. We wish to calculate
$$\frac{\partial}{\partial t}\vol_J[\iota_t]|_{t=0}=\frac{\partial}{\partial t}\sqrt{\vol_{\overline{g}}(e_1(t),\ldots,e_n(t),Je_1(t),\ldots,Je_n(t))}\vol_{g_t}|_{t=0}$$
 at $p$, where $e_1(t),\ldots,e_n(t)$ is orthonormal at $\iota_t(p)\in L_t$ with respect to $g_t=\iota_t^*\bar g$ and $e_i(0)=e_i$ for all $i$.  This computation can be simplified by noticing that
$$\vol_J[\iota_t]=\rho_J(t)\vol_g,$$
where $\rho_J(t)=\sqrt{\vol_{\overline{g}}((\iota_t)_*e_1,\ldots,(\iota_t)_*e_n,J(\iota_t)_*e_1,\ldots,J(\iota_t)_*e_n)}$. Indeed, setting $\pi_t=T_{\iota_t(p)}L_t$ and using the formulae \eqref{rhoJ.eq.3} for $\rho_J$, we find
\begin{align*}
\vol_J[\iota_t]&(e_1,\dots,e_n)=\rho_J(\pi_t) \vol_{g_t}(e_1,\dots,e_n)\\
&=\frac{\sqrt{\vol_{\overline{g}} ((\iota_t)_*e_1,\dots,(\iota_t)_*e_n,J(\iota_t)_*e_1,\dots,J(\iota_t)_*e_n)}}{|(\iota_t)_*e_1\wedge\dots\wedge (\iota_t)_*e_n|_{\bar g}}\cdot|e_1\wedge\dots\wedge e_n|_{\iota_t^*\bar g}\\
&=\sqrt{\vol_{\overline{g}}((\iota_t)_*e_1,\ldots,(\iota_t)_*e_n,J(\iota_t)_*e_1,\ldots,J(\iota_t)_*e_n)}.
\end{align*}
It therefore suffices to compute
\begin{align}
& \frac{\partial}{\partial t}\rho_J(t)=\frac{\sum_{i=1}^n
\vol_{\overline{g}}((\iota_t)_*e_1,\ldots,\frac{\partial}{\partial t}(\iota_t)_*e_i,\ldots,(\iota_t)_*e_n,J(\iota_t)_*e_1,\ldots,J(\iota_t)_*e_n)}
{2\rho_J(t)}\nonumber\\
&+\frac{\sum_{i=1}^n\vol_{\overline{g}}((\iota_t)_*e_1,\ldots,(\iota_t)_*e_n,J(\iota_t)_*e_1,\ldots,\frac{\partial}{\partial t} J(\iota_t)_*e_i,\ldots,J(\iota_t)_*e_n)}
{2\rho_J(t)}.\label{diff.rhoJt.eq}
\end{align}
Setting $t=0$ so that $\frac{\partial}{\partial t}|_{t=0}=\tnabla_Z$  and the denominator becomes $2\rho_J$ we have
\begin{align*}
 \frac{\partial}{\partial t}\rho_J(t)|_{t=0}&=\frac{\sum_{i=1}^n
\vol_{\overline{g}}(e_1,\ldots,\tnabla_Ze_i,\ldots,e_n,Je_1,\ldots,Je_n)}
{2\rho_J}\\
&+\frac{\sum_{i=1}^n\vol_{\overline{g}}(e_1,\ldots,e_n,Je_1,\ldots,\tnabla_Z Je_i,\ldots,Je_n)}
{2\rho_J}.
\end{align*}
Notice that
\begin{align*}
\sum_{i=1}^n\vol_{\overline{g}}(e_1,\ldots,&\tnabla_Ze_i,\ldots,e_n,Je_1,\ldots,Je_n)\\
&=\sum_{i=1}^n\vol_{\overline{g}}(e_1,\ldots,\pi_L\tnabla_Ze_i,\ldots,e_n,Je_1,\ldots,Je_n)\displaybreak[0]\\
&=\sum_{i=1}^n\overline{g}(\pi_L\tnabla_Ze_i,e_i)\vol_{\overline{g}}(e_1,\ldots,e_n,Je_1,\ldots,Je_n)\displaybreak[0]\\
&=\sum_{i=1}^n\overline{g}(\pi_L\tnabla_Ze_i,e_i)\rho_J^2.
\end{align*}
Similarly, using Lemma \ref{proj.lem}
\begin{align*}
\sum_{i=1}^n\vol_{\overline{g}}(e_1,\ldots,&e_n,Je_1,\ldots,\tnabla_ZJe_i,\ldots,Je_n)\\[-8pt]&=
\sum_{i=1}^n\vol_{\overline{g}}(e_1,\ldots,e_n,Je_1,\ldots,\pi_JJ\tnabla_Ze_i,\ldots,Je_n)\displaybreak[0]\\
&=\sum_{i=1}^n\overline{g}(\pi_JJ\tnabla_Ze_i,Je_i)\vol_{\overline{g}}(e_1,\ldots,e_i,\ldots,e_n,Je_1,\ldots,Je_n)\displaybreak[0]\\
&=\sum_{i=1}^n\overline{g}(\pi_L\tnabla_Ze_i,e_i)\rho_J^2.
\end{align*}
Thus,
$$\frac{\partial}{\partial t}\rho_J(t)|_{t=0}=\overline{g}(\pi_L\tnabla_Ze_i,e_i)\rho_J.$$
Since $[Z,e_i]=0$, we have that $$\tnabla_Ze_i=\tnabla_{e_i}Z+\tilde{T}(Z,e_i)+[Z,e_i]=\tnabla_{e_i}Z+\widetilde{T}(Z,e_i),$$
where $\widetilde{T}$ is the torsion of $\tnabla$. The first part of the result follows.

We see that for $Z=X$ tangential
$$\frac{\partial}{\partial t}\vol_J[\iota_t]|_{t=0}=\mathcal{L}_X\vol_J[\iota]=\d(\rho_JX\lrcorner\vol_g),$$
using Cartan's formula, which gives the result.
\end{proof}

Now suppose that $L$ is compact (without boundary).
We can then define the $J$-volume of $L$ as before, using \eqref{rhoJ.eq.2}, by:
$$\Vol_J(L)=\int_L\vol_J=\int_L\rho_J\vol_g=\int_L\sqrt{\vol_{\overline{g}}(e_1,\ldots,e_n,Je_1,\ldots,Je_n)}\vol_g,$$
where $e_1,\ldots,e_n$ is an orthonormal basis for each tangent space.

We now compute the first variation of $\Vol_J$. By Proposition \ref{volJ.prop} if $Z=X+JY$ for $X,Y$ tangential then
\begin{equation}\label{first.var.eq.2}
\frac{\partial}{\partial t}\Vol_J(L_t)|_{t=0}=\int_L
\overline{g}\big(\pi_L\big(\tnabla_{e_i}JY+\widetilde{T}(JY,e_i)\big),e_i\big)\vol_J
\end{equation}
 since 
$\int_L\Div(\rho_JX)\vol_g=0$
by Stokes' Theorem.  Hence it is enough to restrict to $Z=JY$. Our final result is phrased in terms of a ``$J$-mean curvature vector field'' defined as follows.

We use the metric $\bar{g}$ to define the transposed operators
$$\pi_J^t:T_pM\rightarrow (T_pL)^\perp, \ \ \pi_L^t:T_pM\rightarrow (J(T_pL))^\perp.$$
Observe that $(J(T_pL))^\perp=J(T_pL)^{\perp}$ since $X\in (J(T_pL))^{\perp}$ if and only if for all $Y\in T_pL$,
$$\overline{g}(Y,JX)=-\overline{g}(JY,X)=0,$$
which means $JX\in (T_pL)^{\perp}$ and thus $X\in J(T_pL)^{\perp}$.  
Then, using the  tangential projection $\pi_T$ defined using $\overline{g}$, one may check that
$$\pi_T J\tnabla\pi_L^t:T_pL\times T_pL\rightarrow T_pL$$
is $C^{\infty}$-bilinear on its domain, so it is a tensor and its trace is a well-defined vector on $L$. We now set 
\begin{equation}\label{HJ.eq}
H_J:=-J(\mbox{tr}_L(\pi_T J\tnabla\pi_L^t)).
\end{equation}
This is a well-defined vector field on $L$, taking values in the bundle $J(TL)$. We refer to \cite{LP} for an alternative expression for $H_J$.

\begin{prop}\label{first.var.prop}
Let $\iota_t:L\rightarrow L_t\subseteq M$ be compact totally real submanifolds in an almost Hermitian manifold and let 
$\frac{\partial}{\partial t}\iota_t|_{t=0}=X+JY$ for $X,Y$ tangential.  Then 
$$\frac{\partial}{\partial t}\Vol_J(L_t)|_{t=0}=
-\int_L\overline{g}(JY,H_J+S_J)\vol_J$$
where, given $p\in L$ and an orthonormal basis $e_1,\ldots,e_n$ for $T_pL$ we have $H_J$ given by \eq{HJ.eq} and 
\begin{equation}\label{SJ.eq}
S_J=-\overline{g}(\pi_L\widetilde{T}(Je_j,e_i),e_i)Je_j.
\end{equation}
\end{prop}

\begin{proof}  
Using the definition of $H_J$ in \eqref{HJ.eq}, we calculate that
\begin{align*}
\overline{g}\big(\pi_L\big(\tnabla_{e_i}JY&+\widetilde{T}(JY,e_i)\big),e_i\big)\\
&=\overline{g}(\tnabla_{e_i}{JY}+\widetilde{T}({JY},e_i),\pi_L^{\rm t}e_i)\\
&=-\overline{g}({JY},\tnabla_{e_i}\pi_L^{\rm t}e_i)+\overline{g}(\widetilde{T}(\overline{g}(JY,Je_j)Je_j,e_i),\pi_L^{\rm t}e_i)\\
&=-\overline{g}({JY},\tnabla_{e_i}\pi_L^{\rm t}e_i)+\overline{g}(JY,Je_j)\overline{g}(\widetilde{T}(Je_j,e_i),\pi_L^{\rm t}e_i)\\
&=-\overline{g}({JY},\tnabla_{e_i}\pi_L^{\rm t}e_i)+\overline{g}({JY},\overline{g}\big(\widetilde{T}(Je_j,e_i),\pi_L^{\rm t} e_i)Je_j)\\
&=-\overline{g}\Big(JY,-J\pi_{\rm T}J\big(\tnabla_{e_i}\pi_L^{\rm t}e_i-\overline{g}\big(\pi_L\widetilde{T}(Je_j,e_i),e_i)Je_j\big)\Big)\\
&=-\overline{g}\Big(JY,H_J-\overline{g}\big(\pi_L\widetilde{T}(Je_j,e_i),e_i)Je_j\Big),
\end{align*}
where we used $-J\pi_TJ(JX)=JX$. 
The formula follows from \eqref{first.var.eq.2}. 
\end{proof}

\begin{remark}
If we define a Riemannian metric $G$ on $\TRL$ by $$G_L(JX,JY)=\int_L\overline{g}(X,Y)\vol_J$$ for $JX,JY\in T_L\TRL$ then
the downward gradient vector field of the $J$-volume functional $\Vol_J$ on $\TRL$ with respect to $G$ 
is given by $H_J+S_J$. The corresponding flow (the ``$J$-mean curvature flow'') is studied in \cite{LP}.
\end{remark}

\subsection{Second variation of the {\boldmath $J$}-volume}

We now study the stability of critical points of the $J$-volume, so we calculate the second variation of the $J$-volume form.
This generalises calculations in \cite[Proposition 3]{Bor}, which built on the second variation of volume of Lagrangians in 
K\"ahler manifolds derived by Chen, Leung and Nagano \cite[Theorem 4.1]{Chen}.

\begin{prop}\label{second.var.prop}
Let $\iota_{s,t}:L\rightarrow L_{s,t}\subseteq M$ be a two-parameter family of totally real submanifolds in an almost Hermitian manifold and let 
$\frac{\partial}{\partial s}\iota_{s,t}|_{s=t=0}=W$ and $\frac{\partial}{\partial t}\iota_{s,t}|_{s=t=0}=Z$.  Then
\begin{align*}
&\frac{\partial^2}{\partial s\partial t}\vol_J[\iota_{s,t}]|_{s=t=0}\\
&=\Bigg(\overline{g}(\pi_LJ(\tnabla_{e_i}W+\widetilde{T}(W,e_i)),e_j)\overline{g}(\pi_LJ(\tnabla_{e_j}Z+\widetilde{T}(Z,e_j)),e_i)\\
&-\overline{g}(\pi_L(\tnabla_{e_i}W+\widetilde{T}(W,e_i)),e_j)\overline{g}(\pi_L(\tnabla_{e_j}Z+\widetilde{T}(Z,e_j)),e_i)\\
&+\overline{g}\big(\pi_L\big(\tnabla_{e_i}W+\widetilde{T}(W,e_i)\big),e_i\big)\overline{g}\big(\pi_L\big(\tnabla_{e_j}Z+\widetilde{T}(Z,e_j)\big),e_j\big)\\
&+\overline{g}(\pi_L(\widetilde{R}(W,e_i)Z+\tnabla_{e_i}\tnabla_WZ),e_i)\\
&+\overline{g}(\pi_L(\tnabla_W\widetilde{T}(Z,e_i)+\widetilde{T}(\tnabla_WZ,e_i)+\widetilde{T}(Z,\tnabla_{e_i}W+\widetilde{T}(W,e_i))),e_i)
\Bigg)\vol_J[\iota],
\end{align*}
where at $p\in L$ we have that $e_1,\ldots,e_n$ is an orthonormal basis for $T_pL$.
\end{prop}

\noindent Notice again that these quantities are independent of the orthonormal basis chosen for $T_pL$ and hence are globally defined.

\begin{proof}  Let $p\in L$.  We choose coordinates $(x_1,\ldots,x_n,s,t)$ on $L\times(-\epsilon,\epsilon)\times(-\epsilon,\epsilon)$ in a similar manner to the 
proof of Proposition \ref{volJ.prop},
so that $(x_1,\ldots,x_n)$ are normal coordinates at $p$, $e_i=\frac{\partial}{\partial x_i}$ and $\frac{\partial}{\partial s}$, $\frac{\partial}{\partial t}$ pushforward to $W$ and $Z$ 
respectively at $t=0$.  Hence $[W,e_i]=[Z,e_i]=0$, as before.  

Moreover, as in the proof of Proposition \ref{volJ.prop} we see that
$$\vol_J[\iota_{s,t}]=\rho_J(s,t)\vol_{g}$$
where
$$\rho_J(s,t)=\sqrt{\vol_{\overline{g}}((\iota_{s,t})_*e_1,\ldots,(\iota_{s,t})_*e_n,J(\iota_{s,t})_*e_1,\ldots,J(\iota_{s,t})_*e_n)}.$$
By \eq{diff.rhoJt.eq} we have that
\begin{align*}
&\frac{\partial}{\partial t}\rho_J(s,t)^2\\
&=\sum_{i=1}^n\vol_{\overline{g}}((\iota_{s,t})_*e_1,\ldots,\frac{\partial}{\partial t}(\iota_{s,t})_*e_i,\ldots,(\iota_{s,t})_*e_n,J(\iota_{s,t})_*e_1,\ldots,J(\iota_{s,t})_*e_n)
\\
&+\sum_{i=1}^n\vol_{\overline{g}}((\iota_{s,t})_*e_1,\ldots,(\iota_{s,t})_*e_n,J(\iota_{s,t})_*e_1,\ldots,\frac{\partial}{\partial t} J(\iota_{s,t})_*e_i,\ldots,J(\iota_{s,t})_*e_n).
\end{align*}
Hence we may calculate, substituting $\tnabla_W$ for $\frac{\partial}{\partial s}|_{s=t=0}$ and $\tnabla_Z$ for $\frac{\partial}{\partial t}|_{s=t=0}$:
\begin{align}
\frac{\partial^2}{\partial s\partial t}&\rho_J(s,t)^2|_{s=t=0}\nonumber\\
&=\sum_{i\neq j}\vol_{\overline{g}}(e_1,\ldots,\tnabla_Ze_i, \ldots,\tnabla_We_j, \ldots,e_n,Je_1,\ldots,Je_n)\displaybreak[0]\nonumber\\
&+\sum_{i=1}^n\vol_{\overline{g}}(e_1,\ldots,\tnabla_W\tnabla_Ze_i, \ldots,e_n,Je_1,\ldots,Je_n)\displaybreak[0]\nonumber\\
&+\sum_{i,j=1}^n\vol_{\overline{g}}(e_1,\ldots,\tnabla_Ze_i,\ldots,e_n,Je_1,\ldots,\tnabla_WJe_j,\ldots,Je_n)\nonumber\\
&+\sum_{i\neq j}\vol_{\overline{g}}(e_1,\ldots,e_n,Je_1,\ldots,\tnabla_ZJe_i, \ldots,\tnabla_WJe_j, \ldots,Je_n)\nonumber\\
&+\sum_{i=1}^n\vol_{\overline{g}}(e_1,\ldots,e_n,Je_1,\ldots,\tnabla_W\tnabla_ZJe_i,\ldots,Je_n)\nonumber\\
&+\sum_{i,j=1}^n\vol_{\overline{g}}(e_1,\ldots,\tnabla_We_i,\ldots,e_n,Je_1,\ldots,\tnabla_ZJe_j,\ldots,Je_n).\label{second.var.big.eq}
\end{align}
The first and fourth terms in \eqref{second.var.big.eq} both give
\begin{align*}
&\sum_{i\neq j}\big(\overline{g}(\pi_L\tnabla_Ze_i,e_i)\overline{g}(\pi_L\tnabla_We_j,e_j)-\overline{g}(\pi_L\tnabla_Ze_i,e_j)\overline{g}(\pi_L\tnabla_We_j,e_i)\big)\rho_J^2\\
&=\sum_{i,j}\big(\overline{g}(\pi_L\tnabla_Ze_i,e_i)\overline{g}(\pi_L\tnabla_We_j,e_j)-\overline{g}(\pi_L\tnabla_Ze_i,e_j)\overline{g}(\pi_L\tnabla_We_j,e_i)\big)\rho_J^2.
\end{align*}
The second and fifth terms in \eqref{second.var.big.eq} both give
$$\sum_{i=1}^n\overline{g}(\pi_L\tnabla_W\tnabla_Ze_i,e_i)\rho_J^2.$$
Finally, the third and sixth terms in \eqref{second.var.big.eq} both give
\begin{align*}
&\sum_{i,j}\big(\overline{g}(\pi_L\tnabla_Ze_i,e_i)\overline{g}(\pi_L\tnabla_We_j,e_j)-\overline{g}(\pi_J\tnabla_Ze_i,Je_j)\overline{g}(\pi_LJ\tnabla_We_j,e_i)\big)\rho_J^2\\
&=\sum_{i,j}\big(\overline{g}(\pi_L\tnabla_Ze_i,e_i)\overline{g}(\pi_L\tnabla_We_j,e_j)+\overline{g}(\pi_LJ\tnabla_Ze_i,e_j)\overline{g}(\pi_LJ\tnabla_We_j,e_i)\big)\rho_J^2.
\end{align*}

Clearly, by \eq{diff.rhoJt.eq}, we have:
\begin{align*}
\frac{\partial^2\rho_J(s,t)}{\partial s\partial t}=\frac{1}{2\rho_J(s,t)}\left(
\frac{\partial^2\rho_J(s,t)^2}{\partial s\partial t}-2\frac{\partial\rho_J(s,t)}{\partial s}\frac{\partial\rho_J(s,t)}{\partial t}\right),\\ 
\frac{\partial\rho_J(s,t)}{\partial s}\frac{\partial\rho_J(s,t)}{\partial t}|_{s=t=0}=\sum_{i,j}\overline{g}(\pi_L\tnabla_Ze_i,e_i)\overline{g}(\pi_L\tnabla_We_j,e_j)
\rho_J^2,
\end{align*}
and hence
\begin{align*}
&\hspace{-20pt}\frac{\partial^2\rho_J(s,t)}{\partial s\partial t}|_{s=t=0}\\
=&\Big(\sum_{i,j}\overline{g}(\pi_L\tnabla_Ze_i,e_i)\overline{g}(\pi_L\tnabla_We_j,e_j)-\sum_{i,j}\overline{g}(\pi_L\tnabla_Ze_i,e_j)
\overline{g}(\pi_L\tnabla_We_j,e_i)\displaybreak[0]\\
&+\sum_{i,j}\overline{g}(\pi_LJ\tnabla_Ze_i,e_j)\overline{g}(\pi_LJ\tnabla_We_j,e_i)+\sum_{i=1}^n\overline{g}(\pi_L\tnabla_W\tnabla_Ze_i,e_i)\Big)\rho_J.
\end{align*}
Replacing $\tnabla_We_i=\tnabla_{e_i}W+\widetilde{T}(W,e_i)$ (since $[W,e_i]=0$) and similarly for $Z$ gives the first three terms in the statement.  For the remaining terms, we have
\begin{align*}
\tnabla_W\tnabla_Ze_i&=\tnabla_W(\tnabla_{e_i}Z+\widetilde{T}(Z,e_i))\\
&=\widetilde{R}(W,e_i)Z+\tnabla_{e_i}\tnabla_WZ\\
&+(\tnabla_W\widetilde{T})(Z,e_i)+\widetilde{T}(\tnabla_WZ,e_i)+\widetilde{T}(Z,\tnabla_We_i).
\end{align*}
Again replacing $\tnabla_We_i=\tnabla_{e_i}W+\widetilde{T}(W,e_i)$ gives the final terms.
\end{proof}

A simple special case is as follows.

\begin{cor}\label{div.cor}  Use the notation of Proposition \ref{second.var.prop}. 
If $W=Z=X$ tangential then 
$$\frac{\partial^2}{\partial t^2}\vol_J[\iota_t]|_{t=0}=\Div(X\Div(\rho_JX))\vol_g.$$
\end{cor}

\begin{proof} 
We can 
see this from 
$$\frac{\partial^2}{\partial t^2}\vol_J[\iota_t]|_{t=0}=\mathcal{L}_X\mathcal{L}_X\vol_J[\iota]=\d(X\lrcorner\d(\rho_JX\lrcorner\vol_g))$$
using Cartan's formula twice.
\end{proof}

A more important case is the following.

\begin{prop}\label{second.var.Y.prop}  Use the notation of Propositions \ref{first.var.prop} and \ref{second.var.prop}. 
If $W=Z=JY$ where $Y$ is tangential and $M$ is K\"ahler, then 
\begin{align*}
\frac{\partial^2}{\partial t^2}\vol_J[\iota_t]|_{t=0}=&
-\Div(Y\Div(\rho_JY))\vol_g+\left(\frac{\Div(\rho_JY)}{\rho_J}\right)^2\vol_J[\iota]\\
&+\overline{g}(JY,H_{J})^2\vol_J[\iota]-\overline{\Ric}(Y,Y)\vol_J[\iota]\\
&-\overline{g}(\pi_J(\overline{\nabla}_{JY}JY+\overline{\nabla}_{Y}Y),H_{J})\vol_J[\iota]\\
&+\Div\big(\rho_J\pi_L(\overline{\nabla}_{JY}JY+\overline{\nabla}_{Y}Y)\big)\vol_g.
\end{align*}
\end{prop}

\begin{proof}
In this setting $\tnabla=\overline{\nabla}$, the Levi-Civita connection, and $\widetilde{T}=0$.  Hence by Proposition \ref{second.var.prop} we have
\begin{align}
&\frac{\partial^2}{\partial t^2}\vol_J[\iota_{t}]|_{t=0}\nonumber\\
&=\Big(\overline{g}(\pi_LJ\overline{\nabla}_{e_i}JY,e_j)\overline{g}(\pi_LJ\overline{\nabla}_{e_j}JY,e_i)
-\overline{g}(\pi_L\overline{\nabla}_{e_i}JY,e_j)\overline{g}(\pi_L\overline{\nabla}_{e_j}JY,e_i)
\nonumber\\
&\qquad+\overline{g}\big(\pi_L\overline{\nabla}_{e_i}JY,e_i\big)\overline{g}\big(\pi_L\overline{\nabla}_{e_j}JY,e_j\big)\nonumber\\
&\qquad+\overline{g}(\pi_L(\overline{R}(JY,e_i)JY+\overline{\nabla}_{e_i}\overline{\nabla}_{JY}JY),e_i)\Big)\vol_J[\iota].\label{Y.eq-1}
\end{align}
Moreover, Proposition \ref{second.var.prop} and Corollary \ref{div.cor} give us that
\begin{align*}
 &\Div(Y\Div(\rho_JY))\vol_g\\
&=\Big(\overline{g}(\pi_LJ\overline{\nabla}_{e_i}Y,e_j)\overline{g}(\pi_LJ\overline{\nabla}_{e_j}Y,e_i)
-\overline{g}(\pi_L\overline{\nabla}_{e_i}Y,e_j)\overline{g}(\pi_L\overline{\nabla}_{e_j}Y,e_i)
\\
&\qquad+\overline{g}\big(\pi_L\overline{\nabla}_{e_i}Y,e_i\big)\overline{g}\big(\pi_L\overline{\nabla}_{e_j}Y,e_j\big)+\overline{g}(\pi_L(\overline{R}(Y,e_i)Y+\overline{\nabla}_{e_i}\overline{\nabla}_{Y}Y),e_i)\Big)\vol_J[\iota].
\end{align*}

We observe that, since $\overline{\nabla}J=0$ as $M$ is K\"ahler,
\begin{align*}
&\overline{g}(\pi_LJ\overline{\nabla}_{e_i}(JY),e_j)\overline{g}(\pi_LJ\overline{\nabla}_{e_j}(JY),e_i)-
\overline{g}(\pi_L\overline{\nabla}_{e_i}(JY),e_j)\overline{g}(\pi_L\overline{\nabla}_{e_j}(JY),e_i)\\
&=\overline{g}(\pi_L\overline{\nabla}_{e_i}Y,e_j)\overline{g}(\pi_L\overline{\nabla}_{e_j}Y,e_i)-
\overline{g}(\pi_LJ\overline{\nabla}_{e_i}Y,e_j)\overline{g}(\pi_LJ\overline{\nabla}_{e_j}Y,e_i).
\end{align*}
Hence, we see that
\begin{align}
\big(\overline{g}(\pi_LJ\overline{\nabla}_{e_i}(JY),e_j)&\overline{g}(\pi_LJ\overline{\nabla}_{e_j}(JY),e_i)\nonumber\\
&-
\overline{g}(\pi_L\overline{\nabla}_{e_i}(JY),e_j)\overline{g}(\pi_L\overline{\nabla}_{e_j}(JY),e_i)\big)\rho_J\nonumber\\
&=-\Div(Y\Div\rho_JY)+\overline{g}\big(\pi_L\overline{\nabla}_{e_i}Y,e_i\big)\overline{g}\big(\pi_L\overline{\nabla}_{e_j}Y,e_j\big)\rho_J\nonumber\\
&\quad+\overline{g}(\pi_L(\overline{R}(Y,e_i)Y+\overline{\nabla}_{e_i}\overline{\nabla}_{Y}Y),e_i)\rho_J.
\label{Y.eq0}
\end{align}
We see that, using the K\"ahler condition,
\begin{align*}
\overline{g}(\pi_L\overline{R}(JY,e_i)JY,e_i)&=-\overline{g}(\pi_J\overline{R}(JY,e_i)Y,Je_i)\nonumber=-\overline{g}(\overline{R}(Y,\pi_J^{\rm t}Je_i)JY,e_i)\nonumber\\
&=\overline{g}(\overline{R}(Y,\pi_J^{\rm t}Je_i)Y,Je_i)=\overline{g}(\pi_J\overline{R}(Y,Je_i)Y,Je_i).\label{Y.eq2}
\end{align*}
We deduce that
\begin{align*}
\overline{g}(\pi_L\overline{R}(JY,&e_i)JY,e_i)+\overline{g}(\pi_L\overline{R}(Y,e_i)Y,e_i)\\
&=
\overline{g}(\pi_L\overline{R}(Y,e_i)Y,e_i)+\overline{g}(\pi_J\overline{R}(Y,Je_i)Y,Je_i)=-\overline{\Ric}(Y,Y).
\end{align*}
Therefore, using \eq{Y.eq-1} and \eq{Y.eq0} we see that
\begin{align*}
\frac{\partial^2}{\partial t^2}\vol_J[\iota_t]|_{t=0}
=&-\Div(Y\Div(\rho_JY))\vol_g+\overline{g}(\pi_L\overline{\nabla}_{e_i}Y,e_i)^2\vol_J[\iota]\\
&+\overline{g}(\pi_L\overline{\nabla}_{e_i}JY,e_i)^2\vol_J[\iota]-\overline{\Ric}(Y,Y)\vol_J[\iota]\\
&+\overline{g}(\pi_L\overline{\nabla}_{e_i}(\overline{\nabla}_{JY}JY+\overline{\nabla}_{Y}Y),e_i)\vol_J[\iota]
\end{align*}

As in Propositions \ref{volJ.prop} and \ref{first.var.prop}, since here $\widetilde{T}=0$, we see that
\begin{gather*}
\overline{g}(\pi_L\overline{\nabla}_{e_i}Y,e_i)=\frac{\Div(\rho_JY)}{\rho_J},\qquad
\overline{g}(\pi_L\overline{\nabla}_{e_i}JY,e_i)=-\overline{g}(JY,H_J),
\end{gather*}
and
\begin{align*}
\overline{g}(\pi_L\overline{\nabla}_{e_i}(\overline{\nabla}_{JY}JY+\overline{\nabla}_{Y}Y),e_i)=&\frac{\Div\big(\rho_J\pi_L(\overline{\nabla}_{JY}JY+\overline{\nabla}_{Y}Y)\big)}{\rho_J}\\
&-\overline{g}\big(\pi_J(\overline{\nabla}_{JY}JY+\overline{\nabla}_{Y}Y),H_J\big).
\end{align*}
The result now follows.
\end{proof}

We deduce the following important second variation of the $J$-volume functional in the K\"ahler setting, as an immediate corollary of Proposition \ref{second.var.Y.prop}.

\begin{prop}\label{second.var.Kahler.prop}
Let $\iota_t:L\rightarrow L_t\subseteq M$ be compact totally real submanifolds in a K\"ahler manifold and let 
$\frac{\partial}{\partial t}\iota_t|_{t=0}=JY$ for $Y$ tangential.  Then
\begin{align}
\frac{\partial^2}{\partial t^2}\Vol_J(L_t)|_{t=0}=\int_L\Bigg(&\left(\frac{\Div(\rho_JY)}{\rho_J}\right)^2+\overline{g}(JY,H_{J})^2-\overline{\Ric}(Y,Y)\nonumber\\
 &-\overline{g}(\pi_J(\overline{\nabla}_{JY}JY+\overline{\nabla}_{Y}Y),H_{J})\Bigg)\vol_J\label{second.var.Kahler.eq}
\end{align}
\end{prop}

\noindent If $L$ is a critical point of $\Vol_J$, so $H_{J}=0$, then 
all the terms in the integrand in \eqref{second.var.Kahler.eq} are automatically non-negative except for $-\overline{\Ric}(Y,Y)$, so we can ensure non-negativity by imposing an ambient curvature condition.   
We deduce the following, which first appeared in \cite{Bor}.

\begin{cor}\label{stable.cor}  In a K\"ahler manifold with $\overline{\Ric}\leq 0$ (respectively, $\overline{\Ric}<0$), 
 the critical points of the $J$-volume functional are stable (respectively, strictly stable).
\end{cor}

\begin{remark}
We can repeat our discussion in the almost Hermitian setting, but the appearance of torsion terms means the 
second variation formula is more complicated and does not, as far as we are aware, naturally lead to a stability property for the critical points of the 
$J$-volume functional as in the K\"ahler case.
\end{remark}
\subsection{Convexity of the {\boldmath $J$}-volume}\label{ss:convexity}
Stability is an infinitesimal condition. We now want to show that we can obtain a much stronger result by taking into account our notion of geodesics in $\TRL$.  

To start, notice that if $\overline{\Ric}(Y,Y)\leq 0$ then everything in the second variation formula \eqref{second.var.Kahler.eq} is non-negative except potentially for $-\overline{g}(\pi_J(\overline{\nabla}_{JY}JY+\overline{\nabla}_{Y}Y),H_{J})$.  
We also see that, by locally extending $Y$ in a neighbourhood of $L$, 
$$\overline{\nabla}_{JY}JY+\overline{\nabla}_YY=J(\overline{\nabla}_{JY}Y-J\overline{\nabla}_YY)=J[JY,Y].$$
Hence, if we deform $L$ in a direction $JY$ such that $[JY,Y]=0$, which is equivalent to the local diffeomorphisms of $L$ generated by $Y$ and the deformations of $L$ 
generated by $JY$ commute, then $\Vol_J$ is convex in the direction $JY$, in the sense that the second variation is non-negative.  
This condition is precisely guaranteed by our notion of geodesic from Lemma \ref{geod.lem} so we deduce the following.

\begin{thm}\label{thm:convexity}
In a K\"ahler manifold with $\overline{\Ric}\leq 0$ (respectively, $\overline{\Ric}<0$), the $J$-volume functional is convex (respectively, strictly convex) in the sense of Definition \ref{convex.dfn}.
\end{thm}

\subsection{Critical points of the {\boldmath $J$}-volume}\label{ss:critical}

Analogously to the Riemannian setting we  say that a totally real submanifold is \textit{$J$-minimal} if it is a critical point for the $J$-volume, \textit{i.e.}~if $H_J=0$.

Recall from Lemma \ref{Lag.Jvol.lem} that the $J$-volume coincides with the standard volume on Lagrangians and that the sets of $J$-minimal Lagrangians and minimal Lagrangians coincide.  In \cite{LP} we show that this result can be improved by adding assumptions on the ambient manifold. Specifically, we prove the following.

\begin{prop}\label{prop:critical_points}
Assume $M$ is K\"ahler-Einstein with $\overline{\Ric}\neq 0$. Then the sets of $J$-minimal totally real submanifolds and minimal Lagrangians coincide.
\end{prop}

The case of K\"ahler Ricci-flat, in particular Calabi--Yau, manifolds is  special. In this case, by a calibration argument, $J$-minimal totally real submanifolds are $\Vol_J$-minimizers: we call them ``special totally real submanifolds'', in analogy with the well-known class of special Lagrangians, and study them in \cite{LP}.

The following uniqueness result is a simple consequence of strict convexity.

\begin{cor}\label{cor:unique}
Let $M$ be a K\"ahler manifold with $\overline{\Ric}<0$ and $L_0\in\mathcal{T}$ be $J$-minimal. Let $\Gamma\subseteq\mathcal{T}$ denote the set of totally real submanifolds which can be connected to $L_0$ by a geodesic ray. 
Then $L_0$ is the unique $J$-minimal submanifold in $\Gamma$.
\end{cor}

\section{Abstract framework}\label{s:abstract}
We now introduce an abstract framework which will help us clarify and continue to analyze the structure of $\mathcal{T}$.
\subsection{A canonical connection on homogeneous spaces}\label{ss:homog_spaces}
Let $G$ be a finite-dimensional Lie group. Let $L$ and $R$ denote the left and right multiplication operators, $Ad$ the adjoint action of $G$ on $G$ and $ad$ its differential, inducing an action of $G$ on $T_eG$. Let $\mathfrak{g}$ denote the Lie algebra of $G$, \textit{i.e.}~the space of $L$-invariant vector fields. In the course of this section it will be useful to emphasize the distinction between $T_eG$ and $\mathfrak{g}$, using the notation
$$T_eG\rightarrow\mathfrak{g},\ \ X\mapsto\tilde{X}$$
to refer to the isomorphism induced by $L$.

Given $X\in T_eG$, consider the 1-dimensional subgroup of diffeomorphisms of $G$ defined by the flow $\phi_t$ of $\tilde{X}$:
\begin{equation}\label{eq:Xflow}
 \frac{d}{dt}\phi_t=\tilde{X}_{|\phi_t},\ \ \phi_0=\mbox{Id}.
\end{equation}
The isomorphisms $(L_g)_*:T_eG\rightarrow T_gG$ identify each tangent space with $T_eG$, thus inducing a canonical way to differentiate vector fields. This yields a connection on $TG$, known as the \textit{canonical $L$-invariant connection}. The parallel vector fields are the elements of $\mathfrak{g}$, so the flowlines of $\phi_t$ are  the geodesics of this connection. In particular, the geodesics through $e$ are the 1-parameter subgroups $\mbox{exp}(tX):=\phi_t(e)$. The $L$-invariance of the connection implies that $L$ preserves the geodesics. This is reflected in the fact that, for any $g\in G$, $L_g\phi_t$ coincides with the flowline passing through $g$. 

Now fix a closed subgroup $H$. Assume there exists a decomposition
\begin{equation*}
T_eG=T_eH\oplus M, \ \ \mbox{for some $ad_H$-invariant subspace $M$}.
\end{equation*}
Let $\mathfrak{h}$, $\mathfrak{m}$ denote the corresponding $L$-invariant distributions on $G$ so that $TG=\mathfrak{h}\oplus\mathfrak{m}$.
Consider the projection $\pi:G\rightarrow G/H$, viewed as an $H$-principal fibre bundle. Notice that $\mathfrak{h}$ is tangent to the $R_H$-action: choosing  $g\in G$ and a 1-parameter subgroup $h_t$ in $H$, we see that
\begin{equation*}
 \frac{d}{dt}R_{h_t}g_{|t=0}=\frac{d}{dt}gh_{t|t=0}=(L_g)_*\frac{d}{dt}h_{t|t=0}\in \mathfrak{h}_{|g}.
\end{equation*}
This is a manifestation of the fact that $L$-invariant vector fields are the fundamental vector fields of the $R$-action.

We also see   $\mathfrak{m}$ is $R_H$-invariant: given $X\in M$ so that $(L_g)_*X\in\mathfrak{m}_{|g}$, we have
\begin{equation*}
(R_h)_*(L_g)_*X=(L_g)_*(R_h)_*X=(L_g)_*(L_h)_*Y=(L_{gh})_*Y\in\mathfrak{m}_{|gh},
\end{equation*}
where we use that $M$ is $ad_H$-invariant so $(L_{h^{-1}})_*(R_h)_*X=Y$, for some $Y\in M$.

The splitting $TG=\mathfrak{h}\oplus\mathfrak{m}$ thus defines a connection on the principal fibre bundle, and induced connections on all associated bundles $G\times_\rho V$, where $\rho$ is a $G$-action on the vector space $V$.
The following result shows that one of these bundles is particularly relevant to the geometry of $G/H$.
\begin{prop} \label{p:homogeneous}
There is an isomorphism
\begin{equation*}
 G\times_{ad_H}M\rightarrow T(G/H),\ \ [g,X]\mapsto \pi_*(L_g)_*X.
\end{equation*}
Thus $G/H$ has a canonical connection induced from the connection $\mathfrak{m}$ on the principal bundle $G$.
Geodesics in $G/H$ are of the form $\pi(g_t)$, where $g_t$ is a horizontal curve in $G$ satisfying, for some fixed $X\in M$, 
\begin{equation*}
 \frac{d}{dt}g_t=(L_{g_t})_*X\ \ (\mbox{equivalently,}\ \frac{d}{dt}g_t=\tilde{X}_{|g_t}).
\end{equation*}
 In other words, geodesics in $G/H$ are the projections of geodesics in $G$ defined by the canonical $L$-invariant connection, with initial direction in $M$.
\end{prop}

\subsection{Geometry of complexified Lie groups} \label{ss:cpx_lie}
Let $G^c$ be a complexified Lie group, \textit{i.e.}~a complex Lie group with Lie algebra isomorphic to $\mathfrak{g}\otimes \C$. We now study the homogeneous space $G^c/G$.

The maps $L$ and $R$ are holomorphic, so each operator $ad_g:T_eG^c\rightarrow T_eG^c$ commutes with the complex structure $J$ on $G^c$. This implies that $M:=J(T_eG)$ is $ad_G$-invariant, so we can apply the above theory using the splitting $\mathfrak{g}\otimes\C=\mathfrak{g}\oplus i\mathfrak{g}$. According to Proposition \ref{p:homogeneous}, there is an isomorphism 
\begin{equation}\label{Gc.iso.eq}
 G^c\times_{ad_G}(JT_eG)\rightarrow T(G^c/G),\ \ [g,JX]\rightarrow \pi_*(L_g)_*(JX)=\pi_*J(L_g)_*X.
\end{equation}

It follows that $G^c/G$ has a canonical connection, whose geodesics are the projection of curves $g_t$ in $G^c$ satisfying, for some $X\in T_eG$,
\begin{equation}\label{Gc.ODE.eq}
 \frac{d}{dt}g_t=J(L_{g_t})_*X,
\end{equation}
which is an ODE on  $G^c$. If $G^c$ is infinite-dimensional there may be no solutions; however, if a solution does exist for a given initial point, it will exist for any initial point because \eqref{Gc.ODE.eq} is $L$-invariant. In particular, the solution corresponding to the initial point $e\in G^c$ is the 1-parameter subgroup $\expG(tJX)\subset G^c$. 

We can also try to integrate  $X$, obtaining a 1-parameter subgroup $\expG(sX)\subset G$. Assume these subgroups exist. Consider the real 2-dimensional distribution in $TG^c$ generated by $X$ and $JX$. Since the Lie bracket commutes with $J$ we see that $[X,JX]=0$, so the distribution is integrable and our integrations yield a 1-dimensional complex abelian Lie subgroup of $G^c$, spanned by $\expG(sX)$, $\expG(tJX)$. Abstractly, it is the complexification of the Lie group $\expG(sX)$; it is isomorphic to $\Sph^1\times\R$ or to $\C$ depending on 
whether $\expG(sX)$ is compact or not. 

Summarizing, the geodesics in $G^c/G$ are equivalent (through projection and $L$-invariance) to the real 1-parameter subgroups in $G^c$ generated by $JX$, or to the complex 1-parameter subgroups in $G^c$ generated by $X$, for $X\in T_eG$. 

The above applies also to the boundary value problem for geodesics in $G^c/G$: any geodesic $\gamma(t)$, for $t\in [a,b]$, interpolating between  two points in $G^c/G$ lifts to a holomorphic map $\Sigma\rightarrow G^c$, where $\Sigma:=\Sph^1\times [a,b]$ or $\R\times [a,b]$, with prescribed boundary values.
More generally one can study the existence of holomorphic maps $\Sigma\rightarrow G^c$ with given boundary values, where $\Sigma$ is any  Riemann surface with boundary.

\textit{Notation.} From now on we will often relax the distinction between $T_eG$ and $\mathfrak{g}$, and the corresponding distinction between $X$ and $\tilde{X}$.

\begin{definition}
 A function $f:G^c/G\rightarrow\R$ is strictly convex if it is strictly convex when restricted to all geodesics in $G^c/G$. Equivalently, if the lifted function $F:=\pi^*f:G^c\rightarrow\R$ satisfies
 $$JX(JX(F))=\frac{d^2}{dt^2}(F\circ g_t)>0,$$
 for all geodesics $g_t$ in $G^c$ with velocity $JX$, for some $X\in T_eG$.
\end{definition}

\begin{prop}\label{prop:abstract_potential}
 Any strictly convex function $f:G^c/G\rightarrow \R$ lifts to a K\"ahler potential $F:=\pi^*f$ on $G^c$.
\end{prop}
\begin{proof}
 Consider the 2-form $\omega_f:=i\partial\bar{\partial}F=\frac{1}{2}dd^cF$ defined on $G^c$. By construction it is of type $(1,1)$. We need to show that it is positive, \textit{i.e.}~that the symmetric tensor $\omega_f(\cdot,J\cdot)$ is positive definite. 
This is a pointwise statement which must be tested on every vector in $T_gG^c=(\mathfrak{g}+i\mathfrak{g})_{|g}$, for all $g\in G^c$. Equivalently, it suffices to prove that $\omega_f$ is positive when restricted to any complex line. Since $i\mathfrak{g}$ is totally real of maximal dimension, it must intersect the line so we may assume our line is generated by a vector $X$ in $i\mathfrak{g}$. For our computation it is then sufficient to consider the restriction of $F$ to the submanifold of $G^c$ obtained by integrating the vector fields $X,JX$. We now see  our problem corresponds to the $n=1$ case of the following fact: given $f:\R^n\rightarrow \R$ and $F:=\pi^*f:\R^{2n}\rightarrow\R$, if $f$ is strictly convex then $i\partial\bar{\partial}F$ is positive. Indeed, it is simple to compute that 
\begin{gather*}
i\partial\bar{\partial}F=i\frac{\partial^2 F}{\partial z_i \partial \overline{z_j}}dz^i\wedge d\overline{z^j} = 2\sum_{i,j}\frac{\partial^2 f}{\partial x_i \partial x_j}dx^i\wedge dy^j,\\
i\partial\bar{\partial}F((X,Y),(-Y,X)) = 2\frac{\partial^2 f}{\partial x_i \partial x_j}x^i x^j+2\frac{\partial^2 f}{\partial x_i \partial x_j}y^i y^j.
\end{gather*}
Trivially, $\omega_f$ is closed so the result follows.
\end{proof}

Proposition \ref{prop:abstract_potential} shows that any strictly convex function $f$ on $G^c/G$ defines a K\"ahler structure $\omega_f$ on $G^c$. As $G$ acts holomorphically on $G^c$ and preserves the K\"ahler potential,  it preserves $\omega_f$.  Let $\text{Crit}(f)=\{p\in G^c/G:df_{|p}=0\}$ be the set of critical points of $f$.

\begin{prop}\label{prop:abstract_moment_map}
 The action of $G$ on $G^c$, endowed with a K\"ahler structure $\omega_f$, is Hamiltonian with moment map
 \begin{equation*}
 \mu_f:=-\frac{1}{2}d^cF=\frac{1}{2}dF\circ J:G^c\rightarrow\mathfrak{g}^*.
 \end{equation*}
 In particular, $\mu_f^{-1}(0)=\pi^{-1}\text{\emph{Crit}}(f)$ . Since $f$ is strictly convex, $\text{\emph{Crit}}(f)$ is either empty or a unique point, so $\mu_f^{-1}(0)$ is either empty or a unique $G$-orbit in $G^c$.
\end{prop}
\begin{proof}
 For any $X\in \mathfrak{g}$, duality with $\mathfrak{g}^*$ defines a function
 \begin{equation*}
  \mu_f\cdot X:=-\frac{1}{2}i_Xd^cF: G^c\rightarrow\R.
 \end{equation*}
We need to show that $X$ is the Hamiltonian vector field associated to $\mu_f$, \textit{i.e.} $d(-i_Xd^cF)=2i_X\omega_f$. 
Using Cartan's formula, we see
\begin{equation*}
 d(-i_Xd^cF)=-\mathcal{L}_X d^cF+i_Xdd^cF=\mathcal{L}_X(dF\circ J)+2i_X\omega_f.
\end{equation*}
The first term vanishes because both $F$ and $J$ are preserved by the action of $G$.

To conclude, notice that $g\in G^c$ lies in $\mu_f^{-1}(0)$ if and only if $(dF\circ J)_{|g}(X)=dF_{|g}(JX)=0$ for all $X\in\mathfrak{g}$. Since $F$ is $G$-invariant this is equivalent to $dF_{|g}=0$, thus $df_{|\pi(g)}=0$.
\end{proof}

\subsection{Existence of critical points via a stability condition}\label{ss:stability}
The interpretation of critical points of $f:G^c/G\rightarrow\R$ as zeros of a moment map is geometrically interesting but in itself does not bring us closer to understanding whether, in any specific situation, such points exist. If however we can embed $G^c$ with its given structures into a larger K\"ahler manifold $M$, we can sometimes apply the following general framework for studying this existence problem.

Let $(M, \overline{g},J,\overline{\omega})$ be a K\"ahler manifold endowed with a $G$-action preserving $\overline{g}$ and $J$. To simplify matters we assume this action is free. Let us assume that $G^c$ also acts on $M$, preserving $J$: this gives a family of $G^c$-orbits in $M$. As a first step, we are interested in finding situations where each   orbit $\mathcal{O}$ admits a canonical function $F$ such that $\overline{\omega}_{|\mathcal{O}}=i\partial\bar{\partial}F$ as in Proposition \ref{prop:abstract_potential}. An example of this is as follows.

Assume $M$ is polarized, \textit{i.e.}~there is a holomorphic line bundle $L$ over $M$ with a Hermitian metric such that the Chern connection has curvature $\Theta=i\overline{\omega}$. Recall the standard formula for $\Theta$ in terms of a local holomorphic section: $\Theta=\bar{\partial}\partial \log |\sigma|^2$. It follows that $\bar{\omega}=i\partial\bar\partial \log|\sigma|^2$. If the action of $G$ is Hamiltonian, the moment map yields a canonical lift of the infinitesimal action of $G$ to the total space of $L$, cf.~\cite[Section 6.5]{DonKron} for details. Let us assume this integrates to an action of $G^c$. Any point in $L$ then generates a $G^c$-orbit $\tilde{\mathcal{O}}$ which projects to a $G^c$-orbit $\mathcal{O}$ in $M$. Let us think of $\tilde{\mathcal{O}}$ as the graph of a non-vanishing holomorphic section $\sigma:\mathcal{O}\subset M\rightarrow L$. Consider the function $F:=\log |\sigma|^2:\mathcal{O}\rightarrow\R$. By construction $G$ preserves the Hermitian metric so 
$F=\pi^*f$, for 
some $f:G^c/G\simeq\mathcal{O}/G\rightarrow\R$. It turns out that $f$ is strictly convex in our sense; the formula for $\overline{\omega}$ corresponds exactly to the situation of Proposition \ref{prop:abstract_potential}.

In the above situation we say an orbit $\mathcal{O}$ is \textit{stable} if $f$ is proper, so it admits a critical point. Let $M^s$ denote the set of points in $M$ whose corresponding $G^c$-orbits are stable. According to Proposition \ref{prop:abstract_moment_map} there is a 1:1 mapping
$$M^s/G^c\simeq \mu_f^{-1}(0)/G.$$
The key point is that, in specific situations, stability of a given orbit can sometimes be tested using purely holomorphic information on $M$ and the $G^c$-action. We thus get a correspondence between holomorphic and symplectic data on $M$, addressing the existence of critical points of the functions $f$. 

Summarizing: if we can embed our given complexified group $G^c$, endowed with the structure $\omega_f$ defined by a strictly convex function $f:G^c/G\rightarrow \R$, into some K\"ahler $M$ with a $G^c$-action so that it coincides with one of these orbits, then we can hope to test the existence of critical points of $f$ by verifying some type of ``stability condition'' on that orbit.

\subsection{Extension to infinitesimal complexifications}\label{ss:inf_complexification}
Finite-dimensional examples of stability and its relation to existence problems are a classical topic of Algebraic Geometry, related to Geometric Invariant Theory and the Kempf--Ness theorem.

Gauge theory provided the first context in which this abstract framework arose in an infinite-dimensional setting: this is related to the Hitchin--Kobayashi conjecture concerning the existence of Hermitian--Einstein connections on a given Hermitian vector bundle $E$ over a K\"ahler manifold, cf.~\cite{DonKron} for details. In this case $G$ is the group of unitary transformations of $E$, and its complexification $G^c$ is the group of automorphisms of $E$. 

In general when $G$ is infinite-dimensional there does not exist a complexification $G^c$, cf.~\cite{Lempert}. The above theory can thus not be applied. For this reason Donaldson \cite{SKD} introduced a slightly weaker notion of ``infinitesimal complexification'' of  $G$. In this framework we can recover the above results, as follows.

\begin{definition}
 Let $Z$ be a smooth manifold. Assume there exists a vector space $V$ and an injection
 $$V\rightarrow\Lambda^0(TZ),\ \ X\mapsto\tilde{X}$$
 such that the vector fields $\tilde{X}$ define a parallelization of $TZ$, thus $TZ\simeq Z\times V$. Assume further that the space of vector fields $\tilde{X}$ is closed under the Lie bracket on $Z$. We then get an induced Lie bracket on $V$ such that $\widetilde{[X,Y]}=[\tilde{X},\tilde{Y}]$. 
 
The above data defines an \textit{infinitesimal Lie group} $Z$ with Lie algebra $V$.
\end{definition}

\begin{definition}\label{def:inf_complexification}
 Let $(Z,J)$ be a complex manifold. Assume there exists a Lie group $G$ acting freely on the right on $Z$ and preserving $J$. Given $X\in\mathfrak{g}$, let $\tilde{X}$ be the corresponding fundamental vector field: specifically, if $X$ is the infinitesimal deformation of the 1-parameter subgroup $g_t$ then $$\tilde{X}_{|\zeta}:=\frac{d}{dt}(\zeta\cdot g_t)_{|t=0}.$$ This defines an injection $\mathfrak{g}\rightarrow\Lambda^0(TZ)$ preserving the corresponding Lie brackets.
 
 Consider the extended map
 \begin{equation}\label{eq:inf_complexification}
 \mathfrak{g}\otimes\C\rightarrow \Lambda^0(TZ),\ \ X+iY\mapsto \tilde{X}+J\tilde{Y}.
 \end{equation}
 Assume \eqref{eq:inf_complexification} is injective and provides a parallelization of $TZ$. As $G$ preserves $J$, we have $\mathcal{L}_{\tilde{X}}J=0$, \textit{i.e.}~$[\tilde{X},JY]=J[\tilde{X},Y]$ for all $Y\in\Lambda^0(TZ)$. The vanishing of the Nijenhuis tensor implies that also $\mathcal{L}_{J\tilde{X}}J=0$. 
 It follows that the image of the map  \eqref{eq:inf_complexification} is closed under the Lie bracket on $Z$ and that this Lie bracket is $J$-linear, \textit{i.e.}~the image is a complex Lie algebra. Thus \eqref{eq:inf_complexification} defines a complex Lie algebra isomorphism onto its image.
 
 We then say that $Z$ is an \textit{infinitesimal complexification} of $G$.
\end{definition}
Given $Z$ and $G$ as above, we can view $\pi:Z\rightarrow Z/G$ as a principal $G$-bundle. 
The fundamental vector fields $\tilde{X}$ define the ``vertical space'', \textit{i.e.}~the kernel of $\pi_*$. The space of fields $J\tilde{X}$ defines a complementary distribution, which is $G$-invariant because $G$ preserves $J$. In other words, the splitting
$$TZ\simeq \mathfrak{g}\oplus i\mathfrak{g}$$ 
defines a connection on $Z$, thus on all associated bundles. 

A priori there is no adjoint action of $G$ on $i\mathfrak{g}$, because there is no actual group $G^c$ inducing it. We can however define an \textit{ad hoc} action using the adjoint action of $G$ on $\mathfrak{g}$, as follows:
\begin{equation}
 ad_G:G\rightarrow GL(i\mathfrak{g}),\ \ ad_g(iX):=i\,ad_g(X).
\end{equation}
This allows us to define the associated bundle $Z\times_{ad_G}(i\mathfrak{g})$.
\begin{prop}\label{p:homogeneous_bis}
 There is an isomorphism
\begin{equation*}
 Z\times_{ad_G}(i\mathfrak{g})\rightarrow T(Z/G),\ \ [\zeta,X]\mapsto \pi_*(J\tilde{X}_{|\zeta}).
\end{equation*}
\end{prop}
\begin{proof}
 The main issue is to check that the map is well-defined, \textit{i.e.}~that the images of $[\zeta\cdot g, ad_{g^{-1}}X]$ and of $[\zeta,X]$ coincide. It suffices to prove that $\widetilde{ad_{g^{-1}}X}_{|\zeta\cdot g}=(\tilde{X}_{|\zeta})\cdot g$, which is a simple computation.
\end{proof}
\begin{remark}
 When $Z=G^c$ is a standard complexification, this construction coincides with the previous one in \eqref{Gc.iso.eq} because $(L_g)_*X$ is the fundamental vector field $\tilde{X}$ of the right action of $G$ on $G^c$.
\end{remark}
As above, it follows that $Z/G$ has a canonical connection whose geodesics are the projection of curves $\zeta_t$ in $Z$ satisfying, for some $X\in T_eG$,
\begin{equation*}
 \frac{d}{dt}\zeta_t=J\tilde{X}_{|\zeta_t}.
\end{equation*}
As in Section \ref{ss:cpx_lie} this is an ODE: solving it corresponds to integrating the vector field $\zeta_t\mapsto J\tilde{X}_{|\zeta_t}$ in $Z$. This problem is $G$-invariant, but here there is no notion of $G^c$-invariance. We can complexify geodesics by combining them with solutions to  $\frac{d}{ds}\zeta_s=\tilde{X}_{|\zeta_s}$, thus obtaining holomorphic curves in $G^c$.

The analogues of Propositions \ref{prop:abstract_potential} and \ref{prop:abstract_moment_map} continue to hold in this context. 

\section{K\"ahler potentials and cscK metrics}\label{s:pottyK}

In Section \ref{ss:inf_complexification} we mentioned that the ideas of Section \ref{s:abstract} can be usefully applied to gauge theory. A second geometric setting in which this abstract framework proves itself useful is the search for constant scalar curvature K\"ahler (cscK) metrics on a complex manifold $(M,J)$ within a given K\"ahler class $[\omega_0]$. In this case the appropriate Lie group $G$ does not admit a formal complexification, so it is necessary to work with the infinitesimal complexifications described in Section \ref{ss:inf_complexification}. The goal of this section is to provide an overview of this problem so as to emphasize analogies with our main topics: totally real submanifolds, the existence of geodesics and the search for critical points of the $J$-volume. 

Let $M$ be a compact manifold. The space $\Diff(M)$ can be given the structure of an infinite-dimensional Lie group with Lie algebra $\mathcal{X}:=\Lambda^0(TM)$. As for any Lie group,   $T_\zeta\Diff(M)$ is spanned by $(L_\zeta)_*X$ for $X\in\mathcal{X}$, or, by  $(R_\zeta)_*Y=Y\circ\zeta$ for $Y\in\mathcal{X}$. Thus, tangent vectors at $\zeta$ are sections of the pullback bundle $\zeta^*{TM}$. If $(L_\zeta)_*X=(R_\zeta)_*Y$, then $X$ and $Y$ are related by the $ad$-action, which coincides with the standard ``pushforward'' action: $ad_\zeta(X)=d\zeta_{|\zeta^{-1}}(X_{|\zeta^{-1}})$.

If $M$ has a complex structure $J$, $\Diff(M)$ receives an induced complex structure $J(X_{|\zeta}):=(JX)_{|\zeta}$. The smooth structure on $\Diff(M)$ is defined so that the corresponding Lie bracket can be calculated in terms of the Lie bracket on $\mathcal{X}$: it follows that the Nijenhuis tensor of $J$ vanishes on $\Diff(M)$ if this is true on $M$. The right action $R$ preserves $J$ but $L$ does not, so $\Diff(M)$ is a complex manifold but not a complex Lie group.

Now assume $(M,J,\omega_0)$ is K\"ahler. Consider the space $\mathcal{H}$ of K\"ahler structures in the cohomology class defined by $\omega_0$. This is a convex subspace of the space of 2-forms on $M$. According to the $\partial\bar{\partial}$-lemma, any such $\omega$ can be written as
$$\omega=\omega_f:=\omega_0+i\partial\bar{\partial}f,$$
for some $f\in C^\infty(M)$. The potential $f$ is well-defined only up to a constant; we can choose a canonical representative for $f$ using a  normalization functional $I:C^\infty(M)\rightarrow\R$ introduced by Bando and Mabuchi, with the following properties.
\begin{itemize}
 \item There is a 1:1 identification 
 \begin{equation} \label{eq:normalization}
  I^{-1}(0)\simeq \mathcal{H}, \ \ f\mapsto\omega_f.
 \end{equation}
\item The tangent space at $f\in I^{-1}(0)$ is 
$$T_fI^{-1}(0)=\{h\in C^\infty(M): \int_M h\vol_{\omega_f}=0\}.$$
We will alternatively denote this space $C^\infty_{\omega_f}(M)$.\footnote{We refer to \cite{SKD} for details.}
\end{itemize}
 Using this identification, for any $f\in\mathcal{H}$ we define 
$$\mathcal{Q}_f:=\{\zeta\in\Diff(M):\zeta^*\omega_f=\omega_0\}.$$
Let $\mathcal{Q}\subset\Diff(M)$ denote the union of all such $\mathcal{Q}_f$. Consider the right action of the subgroup of Hamiltonian diffeomorphisms $G:=\Ham(M,\omega_0)$ on $\Diff(M)$. Each $\mathcal{Q}_f$ is an orbit of this action so $\pi:\mathcal{Q}\rightarrow \mathcal{H}$ is a principal $G$-bundle.

For any symplectic structure $\omega$ we let $\Ham(\mathcal{X},\omega)$  denote the Lie algebra of $\Ham(M,\omega)$. Its elements are the vector fields $X_h^\omega$ satisfying the equation $dh=\omega(X_h^\omega,\cdot)$, for some function $h:M\rightarrow \R$. In the K\"ahler setting it follows that $X^\omega_h=-J\nabla^\omega h$, where $\nabla^\omega$ is the gradient operator defined by the induced metric $g:=\omega(\cdot,J\cdot)$. We can choose $h$ uniquely by ensuring it belongs to $C^\infty_\omega(M)$. We can then identify the Lie algebra $\Ham(M,\omega)$ with the Lie algebra $C^\infty_\omega(M)$, endowed with the natural Poisson bracket on functions (up to sign).

\begin{lem}\label{lem:Ham_adjoint}
The adjoint action of $\Diff(M)$ on $\mathcal{X}$ satisfies
 $$ad_\zeta(X^{\omega_0}_h)=X_{h\circ\zeta^{-1}}^{\omega_f},$$
 for all $\zeta\in\mathcal{Q}_f$. Furthermore, if $h\in C^\infty_{\omega_0}(M)$ then $h\circ\zeta^{-1}\in C^\infty_{\omega_f}(M)$.
\end{lem}
\begin{proof}
 One can check that the two vector fields coincide when contracted with $\omega_f$. The normalization property of $h\circ\zeta^{-1}$ is a straightforward computation.
\end{proof}

The vertical space of the fibration $\pi$ at a point $\zeta\in\mathcal{Q}_f$ is the subspace $(L_\zeta)_*\Ham(\mathcal{X},\omega_0)$ of $T_\zeta\mathcal{Q}$. It follows from Lemma \ref{lem:Ham_adjoint} that $(L_\zeta)_*\Ham(\mathcal{X},\omega_0)=(R_\zeta)_*\Ham(\mathcal{X},\omega_f)=\Ham(\mathcal{X},\omega_f)_{|\zeta}$.
One can check that there is a splitting
$$T_\zeta\mathcal{Q}=\Ham(\mathcal{X},\omega_f)_{|\zeta}\oplus J\Ham(\mathcal{X},\omega_f)_{|\zeta}.$$
Hence, $\mathcal{Q}$ is an infinitesimal complexification of $\Ham(M,\omega_0)$ in the sense of Definition \ref{def:inf_complexification}. Equivalently, $T\mathcal{Q}$ is parallelized by the map
$$C^\infty_{\omega_0}(M)\otimes\C\rightarrow \Lambda^0(T\mathcal{Q}), \ \ h+ik\mapsto \Big(\zeta\mapsto \big(X^{\omega_f}_{h\circ\zeta^{-1}}+JX^{\omega_f}_{k\circ\zeta^{-1}}\big)_{|\zeta}\Big).$$

We thus learn that $\mathcal{H}=\mathcal{Q}/G$ has a canonical connection. As above, the geodesics are the curves $f_t\subset\mathcal{H}$ obtained by projection of solutions $\zeta_t\subset\mathcal{Q}$ to 
\begin{equation}\label{eq:Q_geodesics}
\frac{d}{dt}{\zeta_t}=J\left(X^{\omega_{f_t}}_{h\circ{\zeta_t^{-1}}}\right)_{|\zeta_t},
\end{equation}
where $h$ is a time-independent function in $C^\infty_{\omega_0}(M)$ and $\zeta_t\in \mathcal{Q}_{f_t}$. As mentioned in Section \ref{ss:inf_complexification} we can also integrate $\frac{d}{ds}\zeta_s=\tilde{X}_{|\zeta_s}:=(X^{\omega_f}_{h\circ\zeta_s^{-1}})_{|\zeta_s}$, thus obtaining holomorphic curves in $\mathcal{Q}$: these are the \textit{complexified Hamiltonian flows} in \cite{SKD}.

According to our identifications, the right-hand side of (\ref{eq:Q_geodesics}) projects to the tangent vector $h\circ\zeta_t^{-1}\in T_{f_t}\mathcal{H}$, so the projected equation is 
\begin{equation}\label{eq:H_geodesics}
 \frac{d}{dt}f_t=h\circ\zeta_t^{-1},
\end{equation}
for some $t$-independent $h$. We can incorporate this condition into the equation by noticing that $h=\frac{d}{dt}f_t\circ\zeta_t=\dot{f}_t\circ\zeta_t$ (notice the change in notation). The geodesic equation on $\mathcal{H}$ is thus
$$\dot{h}=\ddot{f}_{t|\zeta_t}+d\dot{f}_{t|\zeta_t}(\dot{\zeta}_t)=0.$$
We can express the right-hand side of (\ref{eq:Q_geodesics}) as the gradient $\nabla^{\omega_{f_t}}(\dot{f}_t)$. We thus arrive at the final expression for geodesics on $\mathcal{H}$:
\begin{equation}\label{eq:H_geodesicsbis}
 \ddot{f}_{t|\zeta_t}+|\nabla^t\dot{f}_{t|\zeta_t}|_t^2=0,
\end{equation}
where we use $\nabla^t$, $|\cdot|_t$ to indicate   we are using the metric induced by $\omega_{f_t}$. 

\begin{remark}\label{eq:geodesic_comparison}
It is useful to compare (\ref{eq:H_geodesics}) and (\ref{eq:H_geodesicsbis}): the former is first order, expressing the fact that $\dot{f}_t$ coincides with a parallel vector field; the latter is second order and uses  that $C^\infty(M)$ is a vector space, thus has a natural connection. In other words,  (\ref{eq:H_geodesicsbis}) describes  geodesics of the canonical connection on $\mathcal{Q}/G$ in terms of the natural connection on $C^\infty(M)$. 
\end{remark}

We now turn to the problem of finding  cscK metrics in $\mathcal{H}$. It turns out that there is a functional $f:\mathcal{H}\rightarrow\R$, due to Mabuchi, with the following properties:
\begin{itemize}
\item $f$ is convex with respect to the geodesics defined above;\vspace{-4pt}
\item the critical points of $f$ are precisely the potentials of cscK metrics in $\mathcal{H}$.
\end{itemize}
Now consider the space $\mathcal{J}$ of integrable complex structures on $M$ which are compatible with $\omega_0$. Let $G:=\Ham(M,\omega_0)$ act on $\mathcal{J}$ by pullback. It can be shown that $\mathcal{J}$ has a canonical K\"ahler structure which is preserved by the action of $G$. Furthermore, it is possible to embed $\mathcal{Q}$, together with the K\"ahler structure defined by $f$ according to Proposition \ref{prop:abstract_potential}, into $\mathcal{J}$: this is described \textit{e.g.}~in \cite[Chapter 9]{Gau2}, so we do not review it here. As in Section \ref{ss:cpx_lie}, this embedding provides strong geometric motivation for the Yau--Tian--Donaldson conjecture concerning stability conditions related to the existence of a cscK metric in $\mathcal{H}$.

\section{Complexified diffeomorphism groups}\label{s:grandconclusion}
Let $(M,J)$ be a complex manifold. In Section \ref{s:pottyK} we argued that $\Diff(M)$ inherits a complex structure which is formally integrable. The same construction applies to the space of immersions $\mathcal{I}$ of $L$ into $M$; the space $\mathcal{P}$ is then an open subset of $\mathcal{I}$, so it is formally an infinite-dimensional complex manifold. The action of $\Diff(L)$ by reparametrization preserves the complex structure. We can thus reformulate the material of Section \ref{s:totally_real} in terms of the formalism of Section \ref{s:abstract}. We conclude the following.
\begin{itemize}
 \item $\mathcal{P}$ is an infinitesimal complexification of $\Diff(L)$.\vspace{-4pt}
 \item Proposition \ref{p:homogeneous_bis} applies to $\mathcal{P}$, proving that the tangent space $T(\TRL)$ can be identified with the adjoint bundle associated to $\mathcal{P}$.\vspace{-4pt}
 \item The connection and geodesics defined in Section \ref{s:geodesic_eq} coincide with those defined in Section \ref{ss:inf_complexification}.\vspace{-4pt}
 \item When $M$ is K\"ahler and $\overline{\mbox{Ric}}(M)<0$, Proposition \ref{prop:abstract_potential} shows that the functional $\mbox{Vol}_J$ defines a Kahler structure $\omega_J$ on $\mathcal{P}$. Proposition \ref{prop:abstract_moment_map} also applies, showing that the critical points of $\mbox{Vol}_J$ can be interpreted as the zero set of a moment map. 
\end{itemize}

\begin{remark} A theorem of Bruhat and Whitney \cite{WhitneyBru} shows that any real analytic $L$ can be ``complexified'', \textit{i.e.}~embedded as a totally real submanifold into an appropriate complex manifold $(M,J)$. It thus defines a space $\mathcal{P}$. It follows that the corresponding group $\Diff(L)$ admits an infinitesimal complexification even though it may not admit a genuine complexification \cite{Lempert}.
\end{remark} 

A special case of the above occurs when $M$ is negative K\"ahler--Einstein: in this case, using Proposition \ref{prop:critical_points}, we obtain a reformulation of minimal Lagrangians in terms of the zero set of a moment map. The analogies with the theory of cscK metrics and Hermitian--Einstein connections lead to the following question, which seems worthy of further pursuit.

\ 

\noindent\textbf{Question } Can the existence of minimal Lagrangians in negative KE manifolds be related to a stability condition concerning $(M,\overline{g},J,\bar{\omega})$ and the chosen homotopy class $\TRL$?

\ 

In \cite{LPpersist} we study a different existence question: given a minimal Lagrangian for a negative KE metric on $M$, we prove the existence of minimal Lagrangians with respect to small KE perturbations of that metric.

Concerning uniqueness, one can again formulate several different questions depending on the set of submanifolds one chooses to work with. As seen in Corollary \ref{cor:unique}, a minimal Lagrangian $L_0$ is unique within the set of totally real submanifolds which can be connnected to $L_0$ via a geodesic ray: in some sense this is a global statement, but of course it is of interest only in the presence of a good existence theory for geodesics. In \cite{LPpersist} we discuss the question of local uniqueness, \textit{i.e.} within a neighbourhood of a minimal Lagrangian $L_0$ in $\mathcal{T}$. Existence and uniqueness conjectures in the context of Fukaya categories are formulated in \cite{Joy}.

\bibliographystyle{amsplain}
\bibliography{totally}

\end{document}